%% file: LMV.tex
\documentclass{article}
\usepackage[letterpaper,top=2cm,bottom=2cm,left=3cm,right=3cm,marginparwidth=1.75cm]{geometry}
\usepackage[dvipsnames]{xcolor}
\include{StartingPackage}

\usepackage[ruled,vlined]{algorithm2e}
\usepackage{protosem}
\usepackage{authblk}

\providecommand{\keywords}[1]{\small \textbf{Keywords}: #1}

\begin{document}

\title{Line Multiview Varieties}

\thispagestyle{empty}
\date{}

\author[1]{Paul Breiding}
\author[3]{Felix Rydell}
\author[2]{Elima Shehu}
\author[3]{Angélica Torres} 
\affil[1]{{\small University of Osnabr\"uck, Germany}}
\affil[2]{{\small Max Planck Institute for Mathematics in the Sciences, Leipzig, Germany}}
\affil[3]{{\small KTH Royal Institute of Technology, Stockholm, Sweden
}}

\maketitle

\begin{abstract}
We present an algebraic study of line correspondences for pinhole cameras, in contrast to the thoroughly studied point correspondences. We define the line multiview variety as the Zariski closure of the image of the map projecting lines in 3--space to tuples of image lines in 2--space. We prove that in the case of generic camera matrices {the line multiview variety is a determinantal variety} and we provide a complete {set-theoretic} description for any camera arrangement. We investigate basic properties of this variety such as dimension, smoothness, and multidegree. Finally, we give experimental results for the Euclidean distance degree and robustness under noise for the triangulation of lines.
\end{abstract}


\pagenumbering{arabic}

\section*{Introduction} 
Computer vision is a field of artificial intelligence that studies how computers perceive information from images. {A classical problem is structure-from-motion, }
where the task is to create a 3D model of an object from 2D images {taken by unknown cameras}. Such reconstruction problems are central in computer vision with applications to creating models of cities \cite{agarwal2010reconstructing}, modeling clouds \cite{kowsuwan20093d}, and modeling environments for autonomous vehicles \cite{milella20143d}. 
Given a set of $m$ images, the reconstruction process starts by identifying sets of points or lines in one (possibly noisy) image that are recognizable as the same points or lines in another image. These are called \textit{correspondences}. The corresponding points and lines are then used to estimate the positions of the cameras, and reconstruct the original 3D points or lines. The last part of the process is called \textit{triangulation}. 

{Given fixed cameras, the \textit{multiview variety} is the \textit{Zariski closure} of all point correspondences, which means that it is the smallest algebraic variety (i.e., vanishing set of a system of polynomial equations) that contains all point correspondences. Such varieties have, for different camera models, been studied before with tools from algebraic geometry}. In this work, we consider \textit{pinhole cameras:} a projective linear map $C: \mathbb{P}^3 \dashrightarrow \mathbb{P}^2$ defined by a full rank $3\times 4$ matrix~$C$. This camera model{ is the most commonly used camera in state-of-the-art reconstruction algorithms, and is the best understood model from a theoretical point of view. An extensive account of the pinhole cameras is given by \cite{Hartley2004}.} 
\color{black}Ponce, Sturmfels, and Trager \cite{ponce:hal-01423057} introduce geometric cameras as a generalization of pinhole cameras. More recently, in the manuscript~\cite{cid_ruiz2021} Cid-Ruiz, Clark, and Mohammadi develop a nonlinear analog for multiview varieties and compute their multidegrees. For a survey on camera models we refer the reader to~\cite{sturm11}{ and for a survey on algebraic vision as a whole we refer the reader to \cite{kileel2022snapshot}}. 

For an arrangement $\mathcal{C}=(C_1,\ldots ,C_m)$ of $m\geq 2$ pinhole cameras, the map $\Phi_\mathcal{C}:\mathbb{P}^3\dashrightarrow (\mathbb{P}^2)^m$ models the process of taking $m$ images with $m$ cameras. It maps a point $P$ to the tuple of images $(C_1P,\ldots ,C_mP) $ and is defined everywhere except at the camera centers, i.e., the kernels of the $3\times 4$ matrices $C_i$.
A point configuration in $\Phi_\mathcal{C}(\mathbb P^3)\subset (\mathbb{P}^2)^m$ 
is a point correspondence. {The multiview variety $\mathcal{M}_\mathcal{C}$ is the Zariski closure of this image.}
In \cite{Heyden_Astrom}, Heyden and Åström call $\Phi_{\mathcal{C}}(\PP^3)$ the \textit{natural descriptor}, and show that it is not Zariski closed. However, they give a set of polynomial relations that vanish on  $\Phi_{\mathcal{C}}(\PP^3)$, which corresponds to a set of polynomial relations that vanishes also in its Zariski closure, i.e. $\mathcal{M}_\mathcal{C}$. Even if $\mathcal{M}_\mathcal{C}$ is bigger than the natural descriptor, Chevalley's theorem guarantees that the Zariski closure is equal to the Euclidean closure; see \cite[Theorem 4.19]{michalek2021invitation}. The benefit of passing to the Zariski closure is that polynomial expressions describing $\mathcal{M}_\mathcal{C}$ allow for the use of techniques from algebraic geometry.
 
The ideal of $\mathcal{M}_\mathcal{C}$ has been studied in several works since \cite{Heyden_Astrom}. For instance in \cite{agarwal2019ideals} and in \cite{aholt2013hilbert}, where the universal Gr\"obner basis of the ideal is found.
When it comes to geometric properties of $\mathcal{M}_\mathcal{C}$, in \cite{trager2015joint} it is shown that if the cameras are in general position, then $\mathcal{M}_\mathcal{C}$ is smooth; and in \cite{EDDegree_point} a formula for the Euclidean Distance Degree (ED Degree) was provided. All these results give a good understanding of the multiview variety~$\mathcal{M}_\mathcal{C}$.

{In contrast, the algebraic understanding of line correspondences is less extensive, especially for more than three views. Still, line correspondences are of great interest in practice because they appear in abundance in man-made scenarios and are less prone to error than point features in the process of detecting correspondences across images. Moreover, in some real data sets standard feature detection algorithms fail due to a lack of point correspondences, but succeed when line correspondences are taken into account \cite{fabbri2020trplp}. Works such as \cite{Micusik_Wildenauer} and \cite{BARTOLI2005416} explore 3D line reconstruction from the detection and matching of line features in images, to the triangulation and error correction. Specifically, given $n$ views, \cite{Micusik_Wildenauer} uses line segments in three views to build an initial 3D model of a scene, and then adds views successively to recover the scene photographed by the $n$ cameras. Their approach assumes calibrated cameras, and uses the end points of the line segments to check the consistency of the reconstruction when adding views successively, but due to the sensitivity of points to noise, this method can easily run into errors in the reconstruction. In \cite{BARTOLI2005416} they give different methods for line reconstruction where the triangulation and error correction proccesses are based on Plücker coordinates, and they make no assumption on the calibration of the cameras. The use of Plücker coordinates to parametrize lines, although very complete, can be computationally expensive due to the overparametrization of each line. Other examples of the use of lines in different computer vision settings can be found in references such as \cite{Quan,Reisner-Kollmann,line_congruences, Leung, fabbri2020trplp} to name a few. 

Motivated by this we present an algebraic and geometric study of line correspondences in $n$ views for pinhole cameras, with the hope that these results can be used to improve the line triangulation process by including more than 3 views, and providing a description of line correspondences that allows for a better error correction.}

We study the image of the map $\Upsilon_{\mathcal{C}}$ which sends a line~$L$ in $\PP^3$ to the $m$-tuple of lines in $\PP^2$ obtained as the images of~$L$ under the $m$ pinhole cameras of~$\mathcal C$. In symbols:
$\Upsilon_{\mathcal{C}}: \mathbb G \dashrightarrow  (\mathbb P^2 )^m,$
where $\mathbb{G}$ is the Grassmannian of lines in~$\mathbb{P}^3$, and lines in $\mathbb P^2$ are represented by their unique linear equations up to scaling, which gives points in $\mathbb P^2$. {To clarify, throughout this paper we identify $\PP^2$ with its dual $(\PP^2)^\vee$.} The map $\Upsilon_\mathcal{C}$ is defined everywhere except at lines which pass through at least one camera center.
The Zariski closure of~$\Upsilon_\mathcal{C}( \mathbb G)$, denoted by~$\mathcal{L}_\mathcal{C}$, is called the \emph{line multiview variety}, and as in the point case the Zariski and Euclidean closures of~$\Upsilon_\mathcal{C}( \mathbb G)$ are equal. {Our main contribution is to provide a complete set-theoretical description of the variety $\mathcal{L}_{\mathcal{C}}.$} Specifically, we show in Theorem \ref{thm: main1} that
\begin{equation}\label{LC_general_case}
\mathcal{L}_{\mathcal{C}}= \left\{(\ell_1,\ldots,\ell_m)\in (\mathbb P^2)^m \mid \operatorname{rank} \begin{bmatrix}C_1^T \ell_1 & \ldots& C_m^T\ell_m\end{bmatrix}\leq 2\right\},
\end{equation}
if and only if no four camera centers lie on a line, and in Theorem \ref{thm: main2} we explain what happens else. We wish to highlight that the line multiview variety for three general views ($m=3$) had been described in Kileel's PhD thesis \cite{Kileel_Thesis} as part of Theorem 3.10. In fact, this reference covers all possible configurations of points and lines with three cameras. The description provided in \cite{Kileel_Thesis} was a fundamental basis for us to build upon. Equations that are satisfied by three line correspondences have been previously studied in \cite[Section 15]{Hartley2004} and \cite[Section 7]{faugeras1995geometry}. Furthermore, the ideal of {the} critical locus (for which the line reconstruction fails) in the Grassmannian $\mathbb G $, has been computed for three cameras; see \cite{bertolini:hal-01829503}.

Along the description of $\mathcal{L}_\mathcal{C}$, we also show that if the cameras are in general position, then the line multiview variety is smooth as long as $m>3$. In the case $m=3$ there is generally one singular point. As a final contribution{,} we provide a formula for the multidegree of $\mathcal{L}_\mathcal{C}$, and explore its ED degree and sensitivity. 

This paper is structured as follows. In Section~1, we give an overview of the basic mathematical tools we use; experts can safely skip this. In Section 2, we define and describe the line multiview variety. In Section~3, we characterize the possible singularities of the line multiview variety. In Section~4, we compute the multidegree of $\mathcal{L}_\mathcal{C}$, and in Section~5 we give a lower bound for a few of its ED degrees. Finally, in Section 6, we compare errors in triangulation for points and lines from the perspective of numerical analysis.

\bigskip

\noindent\textbf{Acknowledgements} The authors would like to thank Fulvio Gesmundo and Chiara Meroni for helpful discussions and their help in proving Theorem \ref{thm_mdeg}, and Kathl\'en Kohn for simplying the arguments of our main theorem by pointing out a reference to Lemma \ref{le: nongeneric_schubert}, providing us with useful background information and initiating this project. Furthermore, we thank two anonymous referees whose comments greatly improved the paper. The research of Elima Shehu and Paul Breiding was funded by the Deutsche Forschungsgemeinschaft (DFG, German Research Foundation), Projektnummer 445466444. Felix Rydell and Angélica Torres were supported by the Knut and Alice Wallenberg Foundation within their WASP (Wallenberg AI, Autonomous Systems and Software Program) AI/Math initiative.
\section{Preliminaries}
We recall some basic definitions and results from algebraic geometry that we will use in this paper. For completeness, we prove most results in this section. More details can be found in, e.g., \cite{Gathmann}
or \cite{JoeHarris}.

Elementwise complex conjugation of $x\in\mathbb C^{n+1}$ is denoted by $\overline{x}$. Points in $\mathbb C^{n+1}$ are usually understood as column vectors. For $x\in\mathbb C^{n+1}$ we denote by $x^T$ its transpose and by $x^*:=\overline{x}^T$ its conjugate transpose. The Hermitian norm is denoted $\Vert x \Vert := \sqrt{x^*x}$.

The \emph{complex projective space} of dimension $n$ is defined as the set~$\mathbb P^n := (\mathbb C^{n+1}\setminus \{0\})/\sim$ of equivalence classes in $\mathbb C^{n+1}\setminus \{0\}$ given by the relation $x\sim y \Leftrightarrow \exists \lambda \in\mathbb C: x=\lambda y.$
For a complex vector space $V$ we write $\mathbb P(V):=\mathbb P^{\dim(V)-1}$.
For $z=(z_0,z_1,\ldots,z_n)\in\mathbb C^{n+1}\setminus \{0\}$ we write its class as~$[z]=[z_0:z_1:\ldots:z_n]\in \mathbb P^n$, and the projection of $\mathbb C^{n+1}\setminus \{0\}$ onto $\mathbb P^n$ as $\pi:(\mathbb C^{n+1}\setminus \{0\})\to \mathbb P^n$. A subset
$L\subset \mathbb P^n$ is a \emph{$k$-flat}, if $\pi^{-1}(L)\cup\{0\}$
is a $k+1$-dimensional linear space in $\mathbb C^{n+1}$. A $1$-flat in $\mathbb P^n$ is called a \emph{line}, and a $2$-flat is called a \emph{plane}. 

The Hermitian norm on $\mathbb C^{n+1}$ induces a metric on $\mathbb P^n$ via \begin{equation}\label{def_metric_Pn}
d\left([u],[v]\right) := \min_{t\in\mathbb C} \frac{\Vert u-tv\Vert}{\Vert u\Vert}.
\end{equation}
In fact, $d([u],[v])=\sin \alpha$, where $\alpha\in [0,\pi]$ is the minimal angle between two lines in $[u]$ and $[v]$ when interpreted as two-dimensional real vector spaces; see \cite[Proposition 14.12 \& Lemma 14.13]{BC2013}.
For~$x=(x_1,\ldots,x_m), y=(y_1,\ldots,y_m)\in \PP^{n_1}\times \cdots\times \PP^{n_m}$ we set
\begin{equation}\label{def_metric_Pn_product}
d(x,y) := \sqrt{\sum_{i=1}^m d(x_i,y_i)^2}.
\end{equation}
This metric induces a topology on a product of projective spaces, which we call the \emph{Euclidean topology}. For a subset $U$ we denote by $\overline{U}^{\mathrm E}$ its closure in the Euclidean topology.

We denote the ring of complex polynomials in $n+1$ many variables by
$$R:=\mathbb C[x_0,\ldots,x_n].$$
It is a \emph{graded ring}
$R = \bigoplus_{d\geq 0} R_d,$
where $R_d$ denotes the space of homogeneous polynomials of degree~$d$ in $R$.
A subset $X\subseteq \mathbb P^n$ is called a \emph{(projective) algebraic variety}, if there exists a set of \emph{homogeneous} polynomials $f_1,\ldots,f_k\in R$ such that $X = \{x\in\mathbb P^n \mid f_1(x) = \cdots = f_k(x)=0\}$, that is, $X$ is the vanishing set of the~$f_i$. Notice that, in general, for $x\in\mathbb P^n$ and a polynomial $f$ the value $f(x)$ is not defined, but $x$ being a zero of a homogeneous polynomial is well-defined. Similarly, we say that $X\subset \mathbb P^{n_1}\times \cdots\times \mathbb P^{n_m}$
is an algebraic variety, if there exists a set of \emph{multi-homogeneous} polynomials (i.e., homogeneous in each set of variables corresponding to each $\mathbb P^{n_i}$), such that $X$ is their vanishing set. In particular, both $\mathbb P^{n}$ and $\mathbb P^{n_1}\times \cdots\times \mathbb P^{n_m}$ are algebraic varieties. The set of algebraic varieties is closed under intersections and finite unions, so they define the closed sets in a topology, called the \emph{Zariski topology}. The Zariski topology is coarser than the Euclidean topology.

Let $U\subseteq X$ be a subset of a projective algebraic variety $X$. We write
$$I(U)=\{f\in R \mid f(x) = 0\text{ for all } x\in U\}$$
for the homogeneous ideal of polynomials vanishing on $U$.
The \emph{Zariski-closure} $\overline{U}$ of $U$ is the closure of $U$ in the Zariski topology; that is, the smallest algebraic variety containing $U$. We have
\begin{equation}\label{Zariski_closure}
\overline{U} = \{x\in\mathbb P^n \mid f(x) = 0 \text{ for all } f \in I(U)\}.
\end{equation}
Indeed, if $f$ is any polynomial that vanishes on a variety $X$ containing $U$, then, $f$ must also vanish on~$U$, hence $f\in I(U)$. This shows that $\overline{U}\subseteq X$ and moreover that $\overline{U}$ is the variety defined as the zero set of $I(U)$. For an algebraic variety $X$ we call $I(X)$ its \textit{defining ideal} and we denote its \emph{coordinate ring}
$$R[X]: = R/I(X).$$

A variety $X$ is \textit{irreducible}, if and only if for every decomposition $X=Y\cup Z$ into varieties $Y$ and~$Z$ we must have either $X=Y$ or $X=Z$. This is equivalent to the ideal $I(X)$ being a prime ideal.
Following \cite[Definition 2.25]{Gathmann} the \emph{dimension} $m=\dim (X)$ of an irreducible algebraic variety $X$ is the length of a longest chain of irreducible subvarieties
\begin{equation}\label{def_dimension}
\emptyset \neq X_0 \subsetneq X_1 \subsetneq \cdots \subsetneq X_m = X.
\end{equation}
Equivalently, the $m=\dim(X)$ is
the \emph{Krull-dimension} of $R[X]$; i.e.,  the longest chain of prime ideals in $R[X]$ of the form $0= \mathfrak p_m \subsetneq \cdots \subsetneq \mathfrak p_0 \neq R[X]$; see \cite[Lemma 2.27]{Gathmann}. This definition corresponds locally to our intuitive understanding of dimension \cite[Chapter 10]{Gathmann}.

\begin{lemma}\label{lemma_X_equal_to_Y}
Let $X$ and $Y$ be irreducible varieties such that $X\subset Y$, and $\dim (X) = \dim (Y)$. Then, we have~$X=Y$. 
\end{lemma}
\begin{proof}
We prove the assertion by contradiction. If $X\neq Y$, there is a point $x\in Y\setminus X$. By definition of the dimension of a variety, there is a chain of irreducible subvarieties
\[\emptyset \neq X_0 \subsetneq X_1 \subsetneq \cdots \subsetneq X_m = X,\]
where $m:=\dim (X) = \dim (Y)$. Given that $x\not\in X$, we also have the chain
\[\emptyset \neq X_0 \subsetneq X_1 \subsetneq \cdots \subsetneq X_m = X \subsetneq Y,\]
where $X_i$ for $i=0,\ldots ,n$, is irreducible, and $X$ is irreducible by hypothesis. This implies that $\dim (Y) \geq \dim (X)+1$, which contradicts that $\dim (X) = \dim (Y)$.
\end{proof}

\subsection{Regular and rational maps}
Let $U\subset X$ and $W\subset Y$ be subsets of algebraic varieties $X$ and $Y$. A map
$$\varphi: U\to W$$
is \emph{regular} if we can write $\varphi(x)=[\varphi_0(x):\varphi_1(x):\ldots:\varphi_m(x)]$ for polynomials $\varphi_0,\ldots,\varphi_m$. If we have regular maps $\varphi: X\to Y$ and $\psi: Y\to X$  between algebraic varieties such that $\psi \circ \varphi = \mathrm{Id}_X$ and $\varphi\circ \psi = \mathrm{Id}_Y$, we say that $X$ and $Y$ are \emph{isomorphic}. In particular, if $X$ and $Y$ are isomorphic, then~$x\in X$ is a smooth point of $X$, if and only if $\varphi(x)$ is a smooth point of $Y$.  
If $X$ is irreducible, and $\varphi: U \to Y$ is a regular map defined on a dense Zariski open set $U\subset X$, we say that it is a rational map from $X$ to $Y$, denoted by $\varphi: X\dashrightarrow Y.
$
A rational map $\varphi: X\dashrightarrow Y$ is \emph{dominant}, if $Y=\overline{\varphi(X)}$. {If $\varphi$ is a rational map, that is invertible on a dense open subset of $Y$, and if the inverse is again rational, we call $\varphi$ a \emph{birational map}.}

We mention as in \cite[Remark 2.16]{Gathmann} that if $X$ is an irreducible variety, then any non-empty open set is Zariski dense in $X$ and the intersection of two non-empty Zarski open sets is always non-empty. We now prove three lemmata.

\begin{lemma}\label{le: polyzar} Let $\varphi: X\dashrightarrow Y$ be a rational map of irreducible varieties and let $U\subset X$ be non-empty and Zariski open. Then,
$\overline{\varphi(U)}=\overline{\varphi(X)}.$
\end{lemma}
\begin{proof}
Let $V\subset X$ be the open set where $\varphi$ is defined. Since both $U$ and $V$ are non-empty Zariski open in $X$, their intersection $U\cap V$ is as well. Therefore, we can without restriction assume $U\subset V$.

It is clear that $\varphi(U)\subseteq \varphi(X)$, which shows $\overline{\varphi(U)}\subseteq \overline{\varphi(X)}$.
For the other inclusion, it is enough to show $\varphi(V)\subseteq \overline{\varphi(U)}$. Let $x\in V$ and $f\in I(\varphi(U))$. By definition, $f$ is such that $(f\circ \varphi)(u)=0$ for every $u\in U$. Hence, $f\circ \varphi\in I(U)$. Since $I(U)=I(X)$ by Equation (\ref{Zariski_closure}) this shows that $(f\circ \varphi)(x)=0$. Finally, since $I(\varphi(U))=I(\overline{\varphi(U)})$ and since $f$ was arbitrary we have~$\varphi(x)\in \overline{\varphi(U)}$.
\end{proof}

\begin{lemma}\label{image_is_irreducible}
Let $X$ be an irreducible algebraic variety, $Y$ be any variety, and
$\varphi: X\dashrightarrow  Y$
 a rational map. Then, the Zariski closure $\overline{\varphi(X)}$ is an irreducible variety.
\end{lemma}
\begin{proof}
Denote $Z:=\varphi(X)$. By (\ref{Zariski_closure}) we have $I(Z)=I(\overline{Z})$.
Let $f$ and $g$ be polynomials such that $fg\in I(Z)$. We show that either $f\in I(Z)$ or $g\in I(Z)$, which implies that $I(Z)=I(\overline{Z})$ is a prime ideal, hence $\overline Z$ is irreducible.
Let $U\subset X$ be open and dense, such that $\varphi:U\to Z\subset Y$ is a regular map and $\varphi(U)=Z$. Then, $h=(f\circ \varphi)\cdot(g\circ \varphi)$ vanishes on $U$; i.e., $h\in I(U)$.  By~(\ref{Zariski_closure}),  $I(U)=I(X)$. Since $X$ is irreducible, $I(X)$ is prime so we must have either $(f\circ \varphi)\in I(X)$ or~$(g\circ \varphi)\in I(X)$. This implies that either $f\in I(Z)$ or $g\in I(Z)$.
\end{proof}


\begin{lemma}\label{rational_dimension}
Let $X$ be an irreducible variety and $\varphi: X\dashrightarrow Y$ a dominant rational map. Then, $\dim X \geq \dim Y$.
\end{lemma}
\begin{proof}
By Lemma \ref{image_is_irreducible}, $Y$ is irreducible. Let $R[X]$ and $R[Y]$ be the coordinate rings of $X$ and $Y$, respectively.
We have a ring homomorphism $\varphi^*: R[Y]\to R[X]$, called \emph{pull-back} morphism, defined by~$\varphi^*(f):=f\circ \varphi$.
We show that $\varphi^*: R[Y]\to R[X]$ is injective. Let $f\in R[Y]$ with $f\neq 0$. Then, $f$ defines a non-zero function $Y\to\mathbb C$. Since $\varphi$ is dominant, there exists $x\in X$ with $f(\varphi(x))\neq 0$. Therefore, $\varphi^*(f)=f\circ\varphi\neq 0$.
Hence,~$\varphi^*$ defines an embedding $R[Y]\hookrightarrow R[X]$, which implies the Krull-dimension of $R[Y]$ is less or equal than the Krull-dimension of $R[X]$. This shows  $\dim X\geq \dim Y$.
\end{proof}

If $f_1,\ldots,f_k\in R$ are polynomials, we write $\langle f_1,\ldots,f_k\rangle:=\{\sum_{i=1}^kg_if_i\mid g_i\in R\}$ for the ideal generated by the $f_i$. Let $X$ be a variety and suppose that $I(X)=\langle f_1,\ldots,f_k\rangle$. We say that a point $a\in X$ is a \emph{smooth point}, if the rank of the \emph{Jacobian matrix} $J(x):=\big[\tfrac{\partial f_i}{\partial x_j}(a)\big]_{1\leq i\leq k, 0\leq j\leq n}$ is equal to the codimension of $X$. This definition is independent of the choice of generators for the ideal~$I(X)$ \cite[Chapter 10]{Gathmann}. {In our case, however, we only have a set-theoretic description of the ideal of the line-multiview variety. For proving smoothness we use van der Waerden's purity theorem.
\begin{theorem}[Theorem 2.22 of \cite{2002algebraic}]\label{purity_thm} Let $\varphi:X\to Y$ be a birational map between projective complex varieties that is defined on all of $X$. 
Let $\mathcal W:= \{V\subset Y\mid \text{$V$ is open and }\varphi:\varphi^{-1}(V)\to V \text{ is an isomorphism}\}$. 
If $Y$ is smooth, then the union 
$$W:=\bigcup\limits_{V\in\mathcal W} V\subseteq Y$$ 
has the property that $X\setminus \varphi^{-1}(W)$ is either empty or of codimension 1. 
\end{theorem}}


\subsection{The Grassmannian of lines in $\mathbb P^3$}\label{sec: Grassmanians} A particularly important variety for our study is the Grassmannian of lines in $\mathbb P^3$, defined as
$$\mathbb G: = \left\{L\subset \mathbb P^3 \mid L \text{ is a line}\right\}.$$
Every 
element in $\mathbb{G}$ has the form $L = \left\{[x_0s + y_0t: x_1s+y_1t: x_2s+y_2t:x_3s+y_3t] \mid s,t\in\mathbb C\right\}\subset \mathbb P^3$,
where $x=[x_0:x_1:x_2:x_3], y=[y_0:y_1:y_2:y_3]\in\mathbb P^3$ are two distinct fixed points; i.e., $L$ is the line through $x$ and~$y$, and we denote it $L(x,y)$.

The Grassmannian can be given the structure of an algebraic variety through the \emph{Plücker embedding}, which identifies $\mathbb{G}$ with the \textit{Plücker variety} in $\mathbb P ^5$. The Plücker embedding is constructed as follows: For $x,y\in\mathbb P^3$ denote $x\wedge y:=xy^T - yx^T\in \mathbb P(\mathbb C^{4\times 4})$, and define the map 
\begin{equation}\label{def_pluecker}\rho:\mathbb G\to \mathbb P(\mathbb C^{4\times 4}), \; L(x,y) \mapsto x\wedge y.
\end{equation}
To check that this map is well-defined suppose that $L(x,y)=L(u,v)$. Let $0\neq\hat x\in\pi^{-1}(x)$ and similarly define $\hat y, \hat u, \hat v$. Since $L(x,y)=L(u,v)$, $\hat x$ and $\hat y$ span the same two-dimensional vector space as $\hat u$ and $\hat v$. This means that $\hat{u}=\alpha_1 \hat{x}+ \beta_1 \hat{y}$ and $\hat{v}=\alpha_2 \hat{x}+ \beta_2 \hat{y}$ for some scalars $\alpha_1, \alpha_2, \beta_1$ and $\beta_2$ with $\alpha_1\beta_2 - \beta_1\alpha_2\neq 0$, and 
\begin{equation}
    \begin{split}
        \hat u \wedge \hat v = & (\alpha_1 \hat{x}+ \beta_1 \hat{y}) \wedge (\alpha_2 \hat{x}+ \beta_2 \hat{y}) \\
        =& (\alpha_1\beta_2 - \beta_1\alpha_2) \, \hat x \wedge \hat{y}.
    \end{split}
\end{equation}
Projectively we have that $\rho(L(u,v))=u\wedge v= x\wedge y= \rho(L(x,y))$, and $\rho$ is well defined. Additionally, $\rho$ is injective: The column (or row) span of $x\wedge y$ is equal to $L(x,y)$, so if $x\wedge y=u\wedge v$, then $u,v$ span the same line as $x,y$. 

The Plücker embedding gives a bijection between $\mathbb G$ and $\rho(\mathbb G)$, and the latter is the algebraic variety of rank-2 skew-symmetric matrices in $\mathbb{P}(\mathbb C^{4\times 4})$, called the \emph{Plücker variety}. Its defining ideal is 
\begin{equation}\label{pluecker_ideal}
\text{radical ideal of } \Bigg\langle 3\times 3\text{-minors of } \begin{bmatrix}
0 & -p_0 & -p_1 & -p_2\\
p_0 & 0 & -p_3 & -p_4\\
p_1 & p_3 & 0 & -p_5\\
p_2 & p_4 & p_5 & 0
\end{bmatrix} \Bigg\rangle = \langle  p_0p_5-p_1p_4+p_2p_3 \rangle,
\end{equation}
where the equality of the right can be checked, for instance, using \texttt{Macaulay2} \cite{M2}. Therefore, we can interpret the Pl\"ucker variety as a hypersurface in $\mathbb P^5$. The homogeneous coordinates in $\mathbb P ^5$ of each element of~$\mathbb G$ are called \emph{Plücker coordinates}~\cite{JoeHarris}.
In particular, (\ref{pluecker_ideal}) shows that the Grassmanian~$\mathbb G$ is an irreducible hypersurface in $\mathbb P^5$; that is, an algebraic variety of dimension
\begin{equation}\label{dimension_G}
\dim \mathbb G = 4.
\end{equation}
The Zariski open subset of $\mathbb G$, where the Pl\"ucker coordinate $p_1$ is not equal to zero is parametrized by 
\begin{equation}\label{def_tau}\tau:\mathbb C^{2\times 2} \to \mathbb G, \quad \begin{bmatrix} v_{11}&v_{12}\\
v_{21}&v_{22}\end{bmatrix} \mapsto \operatorname{rowspan}\begin{bmatrix} 1&0&v_{11}&v_{12}\\
0&1&v_{21}&v_{22}\end{bmatrix}.
\end{equation}

We define a metric on the Grassmannian as follows. For $L\in \mathbb G$ let $\Pi_L$ denote the orthogonal projection (relative to the Hermitian inner product) onto the two-dimensional linear space $\hat L\subset \mathbb C^4$. The distance between two lines $L, K \in \mathbb G$ is
\begin{equation}\label{def_distance}
\mathrm{dist}(L,K) := \max_{v\in\mathbb C^4\setminus\{0\}: \Vert v\Vert = 1} \Vert \Pi_L(v) - \Pi_K(v)\Vert.
\end{equation}
This distance function induces a topology on $\mathbb G$, which we call the \emph{Euclidean topology}, to distinguish it from the Zariski topology. In the following, when we say that sequences in the Grassmannian converge, we mean convergence with respect to the Euclidean topology. The topology induced by the Euclidean topology in $\PP^5$ gives the same topology on $\mathbb G$.  One interesting property of this metric is unitary invariance. For every unitary matrix $U\in \mathrm{U}(4)$  we have $\mathrm{dist}(L,K) = \mathrm{dist}(U\cdot L, U\cdot K)$. This means that the distance between two lines only depends on their relative position in the ambient space~$\PP^3$.

We also consider the \emph{real Grassmanian} defined by
$$\mathbb G_{\mathbb R}:=\left\{L\in\mathbb G \mid \overline{L}=L\right\},$$
where $\overline{L}$ is the complex conjugate of $L$. The real Grassmannian consists of precisely those lines spanned by real points. Indeed, if $L$ is spanned by real points, it is clearly invariant under conjugation. {On the other hand, if $L$ is invariant under complex conjugation, then suppose that it is spanned by the two points $[a_1],[a_2]\in\PP^3$. It is also spanned by $[\overline{a_1}],[\overline{a_2}]$. We claim that $L$ is spanned by two of the real vectors $a_1+\overline{a_1},a_1-\overline{a_1},a_2+\overline{a_2}$ and $a_2-\overline{a_2}$. These vectors are all contained in $L$ by assumption. One can check that if all of these four vectors were parallel, then so would $a_1$ and $a_2$ be, which is a contradiction.} 

The Zariski closure of $\mathbb R^{2\times 2}$ is $\mathbb C^{2\times 2}$ and so by Lemma~\ref{le: polyzar} we have $\overline{\tau(\mathbb R^{2\times 2})} = \mathbb G.$ Moreover, $\tau(\mathbb R^{2\times 2}) \subseteq \mathbb G_{\mathbb R}$. Together this implies
\begin{equation}\label{Zariski closure real Grassmannian}
\overline{\mathbb G_{\mathbb R}} = \mathbb G.
\end{equation}

We can identify $\CC^4$ with its 
dual space $(\CC^4)^*:=\{f:\CC^4\to \CC\mid f \text{ is linear}\}$ either by using the Hermitian inner product $(x,y)\mapsto x^*y$ or with the Euclidean bilinear form $(x,y)\mapsto x^Ty$. Both options define a notion of \emph{dual line} in $\PP^3$. The dual lines of a line $L\in \mathbb G$ are denoted
\begin{align}\label{def_dual_lines}
L^\perp &:=\{p\in \PP^3 \mid p^Tq=0 \text{ for all } q\in L\}, \quad\text{and}\\
L^* &:=\{p\in \PP^3 \mid p^*q=0 \text{ for all } q\in L\}.\nonumber
\end{align}
Notice that $L^* = \overline{L}^\perp$ and that in the real Grassmannian these two definitions coincide. 

We say that a line $L$ is a \emph{transversal} of another line in $\PP^3$, if $L$ intersects this line. We call the lines intersecting $L$ its transversals.
A fixed line in $\PP^3$ defines the following irreducible subvariety, called the \emph{Schubert variety} of transversals of $L$:
$$
 \Omega(L) = \{W\in \mathbb G \mid W\cap L\neq \emptyset \};
 $$
see \cite[Section 3.3]{eisenbud-harris:16}. If we have four such Schubert varieties defined by four lines in general position $L_1, L_2, L_3, L_4$, then their intersection is finite and
\begin{equation}\label{intersection_Schubert}
\#\,(\Omega(L_1) \cap \Omega(L_2) \cap\Omega(L_3) \cap\Omega(L_4)) = 2;
\end{equation}
see, e.g., \cite[Section 3.4.1]{eisenbud-harris:16}. The next lemma partly explains what generic means in this case. Recall that a \textit{quadric surface} in $\PP^3$ is an algebraic variety defined as the solution set to a single homogeneous polynomial of degree 2 in 4 variables.

\begin{lemma}\label{le: nongeneric_schubert} If $L_1,L_2,L_3,L_4\in\mathbb G$ are four disjoint lines in $\PP^3$, then either
\begin{enumerate}
\item all four lie on a smooth quadric surface, or, 
\item they do not lie on any quadric, and they have
(at most) two common transversals.
\end{enumerate}
\end{lemma}
\begin{proof}
See \cite[Lemma 6.16]{stevensmma320}.
\end{proof}
It is an open condition that four lines in $\PP^3$ are disjoint, 
and we wish to understand what happens if some of the lines intersect. In particular, we need to characterize when there are infinitely many lines intersecting $k$ given lines $L_1,\ldots,L_k$. Equivalently, when the intersection of Schubert varieties $\bigcap_{i=1}^k\Omega(L_i)$ is positive dimensional. The answer is in Lemma \ref{le: inf_trans} below. For the proof, we need yet another lemma.

Let $Q$ be a quadric surface in $\PP^3$, defined by the vanishing of the homogeneous degree 2 polynomial~$q$. There is a unique $4\times 4$ symmetric matrix $A$ such that
$$q(x)=x^TAx.$$
Note that the gradient of $q$ is equal to $2Ax$. This implies that the quadric surface defined by $q$ is smooth if and only if $A$ is an invertible matrix because $2Ax=0$ implies $x^TAx=0$. A quadric containing a plane either has $\operatorname{rank} A=1$ in which case it is a double plane or $\operatorname{rank} A=2$ in which case it is a union of two distinct planes. If $\operatorname{rank} A=3$, then the surface is a cone. A full rank matrix $A$ can via linear coordinate change over $\CC$ be transformed into any other full rank symmetric matrix. This means that all smooth quadrics in $\PP^3$ are isomorphic. They especially differ by the linear coordinate change to the surface $Q$ defined by $q=x^TAx=x_0x_3-x_1x_2=0$, given by the matrix
$$A=\begin{bmatrix}
0 & 0 & 0 & 1\\
0 & 0 & -1 & 0\\
0 & -1 & 0 & 0 \\
1& 0 & 0 & 0 
\end{bmatrix}.$$

Next, we specialize the above result with two lemmas for the proof of our main theorem.
\begin{lemma}\label{le: cont_fam_lines} A smooth quadric $Q$ in $\PP^3$ consists of two continuous 1-dimensional families of lines. More precisely, any point $p\in Q$ meets exactly two distinct lines $L_1(p),L_2(p)\subseteq Q$, one from each family, and as $p$ moves continuously, so does $L_1(p),L_2(p)$. Every line from one family meets every line from the other family.
\end{lemma} 

\begin{proof} By linear transformation, it is enough to prove the statement for the smooth quadric $Q$ defined by $x_0x_3-x_1x_2=0$. Let $[y_0:y_1:y_2:y_3]\in Q$, then the two lines spanned by the row vectors of the matrices
$${\begin{bmatrix} y_0&0&y_2&0\\
y_1&0&y_3&0 \\ {0} & {y_0} & {0} & {y_2}\\  {0} & {y_1} & {0} & {y_3}
\end{bmatrix}},\quad {\begin{bmatrix}
y_0&y_1&0&0\\
y_2&y_3&0&0 \\ {0} & {0} & {y_0} & {y_1}\\  {0} & {0} & {y_2} & {y_3}
\end{bmatrix} },$$
lie in $Q$ (note that both matrices are of rank 2 {and contain $[y_0:y_1:y_2:y_3]$ in their row span}). As $[y_0:y_1:y_2:y_3]$ changes continuously, so do the two lines. In this way, $Q$ consists of two continuous 1-dimensional families of lines. To see that there are no other lines through $[y_0:y_1:y_2:y_3]$, consider for instance an affine patch containing this point, say $y_0=1$. Write $y=(1,y_1,y_2,y_3)$ and consider a line $\{y+tv: t\in \CC\}\subseteq \CC^4$ for some $v=(0,v_1,v_2,v_3)\neq 0$. Note that setting $q(y+tv)=0$ for all $t$ gives two equations in $v_1,v_2,v_3$, one linear and one quadratic. Up to scaling, we get either at most two solutions for $v$ or infinitely many. In the case of infinitely many solutions, the surface $Q$ contains a plane and cannot be smooth. Using the explicit description of the two families of lines above, it can be directly checked that every line from one family meets every line from the other.
\end{proof}

\begin{figure}
    \centering
    \includegraphics[width = 0.75\textwidth]{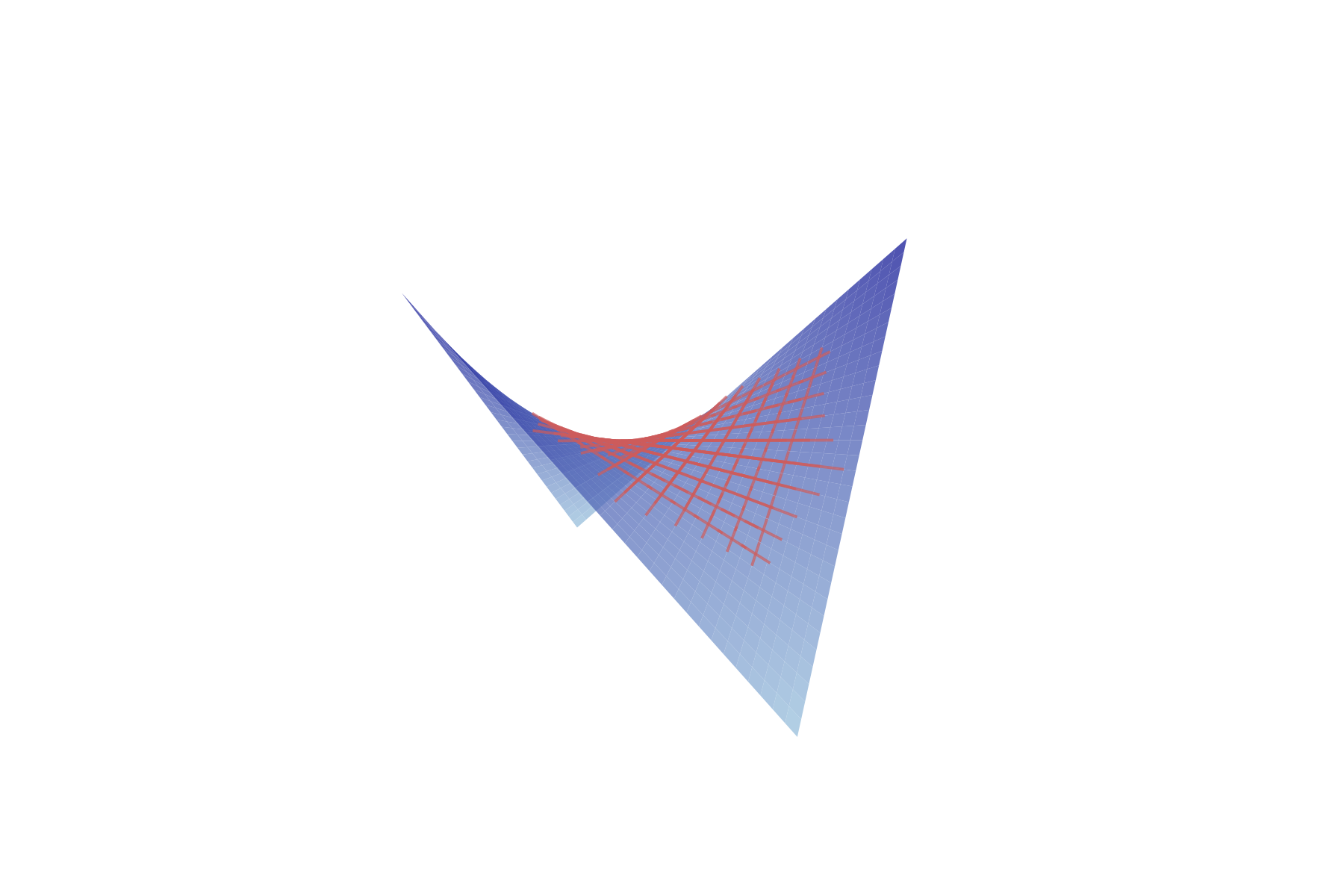}
    \caption{The pictures shows the quadric surface $x_0x_1-x_2x_3=0$ in the affine patch $x_0=1$. The red lines lie on the surface. The pictures were created using \texttt{Plots.jl} \cite{plots}.}
    \label{fog0}
\end{figure}

\begin{lemma}\label{le: inf_trans} Let $L_1,\ldots,L_k\in \mathbb G$ be $k$ lines in $\PP^3$. These lines have infinitely many common transversals if and only if they have three common transversals.
\end{lemma}
\begin{proof} 
 If there are infinitely common transversals, then there are three. So, we need to show that if there are three distinct lines $O_1,O_2,O_3$ intersecting each $L_i$, then there are infinitely many. If $k\leq 3$, there are always infinitely many common transversals, because each Schubert variety $\Omega(L_i)$ is a hypersurface, so $\bigcap_{i=1}^k\Omega(L_i)$ has codimension at most 3 in the 4-dimensional variety $\mathbb G$, hence is positive dimensional.  Assume $k\ge 4$. We consider three different cases. 

The first case is when all lines $L_i$ meet in a point $q$. Then each line through $q$ is a common transversal.

In the second case the first $s\geq 2$ lines $L_1,\ldots,L_s$ meet in a point $q$, and $q\not \in L_{s+1},\ldots, L_k$. We consider two subcases: 
If $O_1$ and $O_2$ meet $q$, then, since they meet in a point, they span a plane $P$. The last $k-s$ lines~$L_{s+1},\ldots, L_k$ intersect $O_1$ and $O_2$ simultaneously, in other words, each of them meets the plane in two distinct points. Therefore, $L_{s+1},\ldots, L_k\subset P$, and so every line in $P$ through $q$ is a transversal of $L_i$ for every $i$. If $O_1$ and $O_2$ do not meet $q$, then we consider the plane $P'$ spanned by $L_1$ and $L_2$. Both $O_1$ and $O_2$ have two intersection points with $P'$, so $O_1,O_2\subset P'$. This implies that $O_1$ and $O_2$ meet in a point $q'\in P'$. Each of the $L_i$ must either meet $q'$ or be contained in $P'$. Therefore, every line in $P'$ through $q'$ is a transversal of $L_i$ for every $i$. In both cases, there are infinitely many common transversals. {If $O_1$ meets $q$ and $O_2$ does not, then either $O_1,O_3$ or $O_2,O_3$ fall under one of the two subcases above.}

Finally, we have the case where the $L_i$ are pairwise disjoint. We have that $L_1,L_2,L_3,L_4$ lie on a smooth quadric $Q$ by Lemma \ref{le: nongeneric_schubert}. Since each $O_i$ intersects $L_1,L_2,L_3,L_4$ in different points, $O_i$ intersects $Q$ in at least 4 points. But then $O_i$ must be contained in $Q$, because the restriction of a degree 2 polynomial to a line gives a univariate polynomial of degree 2, which either has at most two solutions or is constant and equal to zero. Therefore, we have $O_1,O_2,O_3\subseteq Q$. By Lemma \ref{le: cont_fam_lines}, $O_1,O_2,O_3$ are part of the same family of the two families of lines on $Q$. The lines $L_1,\ldots, L_4$ are therefore part of the other family and so, there is a family of lines intersecting each $L_i$.
\end{proof}

Finally, we also need the following lemma for our proofs in the next section.
\begin{lemma}\label{le: three lines} Any three lines in $\PP^3$ lie on a quadric. If the lines are disjoint, the quadric is smooth and unique. 
\end{lemma}
\begin{proof}
If three points of a line lie on a quadric surface, then the whole line must lie on it. This is because the restriction of a degree 2 polynomial to a line gives a univariate polynomial of degree 2, which either has at most two solutions or every point is a solution. Take nine distinct points, three from each line. A quadric in $\PP^3$ is determined by ten coefficients, and nine linear constraints on these imply that there is at least one solution.

Assume that three disjoint lines $L_1,L_2,L_3$ lie on two quadric surfaces $Q,Q'$. We show that $Q=Q'$ and that $Q$ is smooth. First, we show smoothness: Two lines out of any three lines in a plane or union of two planes must meet, and in a cone, any two lines meet. {Identifying the quadric $Q$ with its matrix, recall that a plane corresponds to $\mathrm{rank} \;Q=1$, a union of two planes correspond to $\mathrm{rank} \;Q=2$ and a cone corresponds to $\mathrm{rank} \;Q=3$. By process of elimination, the matrix of $Q$ must have rank 4, and we have seen that this implies that it is a smooth quadric.} Now, we show uniqueness. Assume there is a point $x\in Q'\setminus Q$. By Lemma \ref{le: cont_fam_lines}, $L_1,L_2,L_3$ are from the same family of lines in both $Q,Q'$ and especially, in $Q'$ there is a line $L$ passing through each $L_1,L_2,L_3$ and $x$. But since the three distinct intersection points between $L$ and $L_i$ also lie in $Q$, the line $L$ must lie in $Q$ showing that we have~$x\in Q$.
\end{proof}

\section{Line Multiview Varieties}\label{s: line_mult} 
A pinhole camera is a linear map $\mathbb{P}^3\dashrightarrow \mathbb P^2,$ $x \mapsto C x,$ defined by a full rank $3\times 4$ matrix~$C\in\mathbb C^{3\times 4}$. It induces the following camera map for lines
\begin{equation}\label{def_cam_map}
\mathbb{G}\dashrightarrow\mathbb P^2,\quad L \mapsto \ell = C \cdot L,
\end{equation}
which maps the line $L(x,y)$ to the line through $Cx,Cy$ in $\PP^2$. We identify a line in $\PP^2$ with its linear equation~$\ell$, which is a point in $\PP^2$. That is,~$x\in \ell$ if and only if $x^T\ell = 0$. In fact, $\ell$ is the kernel of the rank-2 matrix $C\rho(L)C^T\in\mathbb C^{3\times 3}$, so (\ref{def_cam_map}) is a rational map.
The kernel of the camera matrix
$$c=\ker(C) \in\mathbb P^3$$
is called the \textit{camera center}. The map (\ref{def_cam_map}) is defined outside the variety of lines passing through $c$.
For every image line, $\ell\in\PP^2$ we have that $C^T\ell\in \PP^3$ defines the plane that is projected onto $\ell$ by~$C$.
To see this, let~$p\in \PP^3$. Then~$p$ is projected onto the line $\ell$ if and only if~$(Cp)^T\ell=0$, which is equivalent to~$p^T(C^T\ell)=0$. The map that sends $C^T\ell$ back to $\ell$ is given by the pseudo-inverse matrix $(C^T)^\dagger$ of~$C^T$, which has the property that~$(C_{i}^{T})^{\dagger}C_{i}^{T}=\mathrm{Id}_{\mathbb C^3}$.

Let $m\geq 2$ and $\mathcal{C}=\{C_1,\dots,C_m\}$ be an arrangement of $m$ pinhole cameras  with different centers. This is our assumption throughout this paper. 
We use the notation  $\ell=(\ell_1,\ldots,\ell_m)\in(\PP^2)^m$, 
$$h_i := C^T_i\ell_i \quad \text{ and } H_i=\left\{p\in \PP^3\mid p^Th_i=0\right\}.$$
We call the plane $H_i$ the \emph{back-projected plane} of the image line $\ell_i$. As pointed out above, the back-projected plane $H_i$ is the plane that projects to $\ell_i$ under the camera matrix $C_i$. This is the geometric interpretation we always keep in mind. For every subset of indices $I\subseteq \{1,\ldots,m\}$, $|I|\geq 2$, we denote the span of the camera centers with index in~$I$ by
\begin{equation}\label{def_E_I}
E_I := \operatorname{span}\{c_i \mid i\in I\}.
\end{equation}
We say that the camera centers (or simply cameras) indexed by $I$ are \emph{collinear}, if $E_I$ is a line called the \textit{baseline} of the centers indexed by $I$. We say that they are \emph{coplanar}, if $E_I$ is a plane.

We consider the \textit{joint camera map}
\begin{equation}\label{def_camera_map}
    \Upsilon_{\mathcal C}: \mathbb G \dashrightarrow (\mathbb P^2)^m,\quad
    L\mapsto (C_1\cdot L,\ldots, C_m\cdot L),
\end{equation}
which sends a line in 3-space to the lines in the image of the cameras, meaning its projections with respect to the camera matrices $C_i$. Observe that $\Upsilon_{\mathcal C}^{-1}(\ell)$ consists of the lines contained in $H_1\cap\cdots \cap H_m$ {meeting no center}.

In this section, we characterize in full generality the \textit{line multiview variety} $$\mathcal{L}_{\mathcal{C}}:=\overline{\Upsilon_{\mathcal{C}}(\mathbb G)},$$
defined as the Zariski closure of the image of the joint camera map. This variety was described for three cameras whose centers are linearly independent in \cite[Theorem 3.10]{Kileel_Thesis}. The line multiview variety is also the Euclidean closure of $\Upsilon_{\mathcal{C}}$. This is implied by Chevalley's theorem; see \cite[Theorem 4.19]{michalek2021invitation}.
For a tuple  $x=(x_1,\ldots,x_m) \in (\mathbb C^3)^m$ we denote the matrix
$$M(x) =  \begin{bmatrix}C_1^Tx_1 & \cdots & C_m^Tx_m\end{bmatrix} \in\mathbb C^{4\times m}.$$
Notice that the rank of this matrix only depends on the projective classes of the $x_i$. For a tuple $\ell=(\ell_1,\ldots,\ell_m)\in(\PP^2)^m$ the $i$-th column $h_i=C_i^T\ell_i$ of $M(\ell)$ defines the back-projected plane $H_i$. If the~$\ell_i$ are images of a joint line $L\subset \PP^3$, the $H_i$ meet in $L$, and so $M(\ell)^Tp=0$ for all $p\in L$. Consequently, the kernel of $M(\ell)^T$ contains two linearly independent vectors, meaning that the rank of $M(\ell)$ is at most 2. The back-projected planes meet in exactly a line when the rank of $M(\ell)$ is equal to 2. If the rank of $M(\ell)$ is 1, the back-projected planes meet in a plane. Theorem \ref{thm: main1} below shows that under natural conditions these rank conditions completely characterize the line multiview variety. Before we state this theorem, however, let us first inspect some basic properties of the line multiview variety. The proofs of these properties are presented in Section \ref{sec: proof basics} below.

\begin{theorem}\label{thm_dimension}
The line multiview variety $\mathcal{L}_{\mathcal{C}}$ is an irreducible variety of dimension $4$.
\end{theorem}
A consequence is that the multiview variety of two cameras with different centers is equal to $\PP^2\times \PP^2$, since this is the only irreducible variety of dimension four inside $\PP^2\times \PP^2$. This also makes intuitive sense; two back-projected planes always meet in at least a line $L$, and it's an open condition for this~$L$ to be projected to the original image lines.

If the cameras are given by real matrices, one may wonder if the equations for $\mathcal{L}_{\mathcal{C}}$ already define the locus of real images $\mathcal{L}_{\mathcal{C}}^{\mathbb R}:=\left\{\ell\in \mathcal{L}_{\mathcal{C}}\mid \overline{\ell}=\ell\right\}$ (consisting of those tuples of lines which are fixed by complex conjugation). The next theorem shows that this is indeed true and that the ideal of polynomial equations vanishing on $\mathcal{L}_{\mathcal{C}}^{\mathbb R}$ is the ideal of $\mathcal{L}_{\mathcal{C}}$.
\begin{theorem}\label{thm_real_dimension}
Suppose that the camera matrices $C_i\in\mathbb R^{3\times 4}$ are real matrices. Then, the real line multiview variety $\mathcal{L}_{\mathcal{C}}^{\mathbb R}$ is Zariski dense in $\mathcal{L}_{\mathcal{C}}$, and the smooth points in $\mathcal{L}_{\mathcal{C}}^{\mathbb R}$ form a smooth manifold of real dimension $4$.
\end{theorem}
The difference between the line multiview variety and the image of the joint camera map is discussed in the next proposition.

\begin{proposition}\label{prop joint image}
We have $\Upsilon_{\mathcal{C}}(\mathbb G) = \mathcal{L}_{\mathcal C}\setminus (\mathcal{X}\setminus \mathcal{Z})$, where
    \begin{align*}
        \mathcal{X}&:=\left\{ \ell\in (\PP^2)^m \mid c_i^TM(\ell)=0 \text{ for some } i\right\} \quad \text{and}\quad
        \mathcal{Z}:=\left\{\ell\in (\PP^2)^m \mid \operatorname{rank} M(\ell)=1 \right\}.
    \end{align*}
In other words, to obtain the image, we remove all image lines whose back-projected planes meet in exactly a line that goes through a camera center.
\end{proposition}

More basic properties of the line multiview variety are presented in the next proposition.
\begin{proposition}\label{properties_of_LMV} Let $\mathcal{C}$ be a collection of cameras. 
\begin{enumerate}
\item $\Upsilon_{\mathcal C}$ is generically identifiable: for all $L\in\mathbb G$ with $\operatorname{rank} M(\ell) =2$, where $\ell =\Upsilon_{\mathcal C}(L)$,  we have~$\Upsilon_{\mathcal C}^{-1}(\ell) = \{L\}$. If $\operatorname{rank} M(\ell)=1$, then $\Upsilon_{\mathcal C}^{-1}(\ell)$ contains infinitely many lines.
\item Let $\mathcal{C}'\subseteq \mathcal{C}$ be a subcollection of cameras with indices $I$, $\vert I\vert \geq 2$. Let $\pi$ be the projection from $\mathcal{L}_{\mathcal{C}}$ to the factors corresponding to the indices of $I$, then
$\pi(\mathcal{L}_{\mathcal{C}})=\mathcal{L}_{\mathcal{C}'}.$
\end{enumerate}
\end{proposition}

Let us now move towards the main theorems of our paper, Theorems \ref{thm: main1} and \ref{thm: main2}, {which characterize set-theoretically the line multiview variety} $\mathcal{L}_{\mathcal{C}}$.

\begin{theorem}\label{thm: main1} Let $\mathcal{C}$ be a collection of $m$ cameras with distinct centers. We have
$$\mathcal{L}_{\mathcal{C}}= \left\{\ell \in (\mathbb P^2)^m \mid \mathrm{rank}\ M(\ell) \leq 2\right\}$$ 
if and only if no four cameras are collinear. This is precisely when the variety on the right-hand side is irreducible and of dimension $4$.
\end{theorem}

\begin{remark}
{The rank condition on the right of this theorem defines an ideal $I$. Theorem \ref{thm: main1} does not imply $I=I(\mathcal{L}_{\mathcal{C}})$. We will deal with the question whether or not this is true in a follow-up paper.}
\end{remark}

\begin{remark} The \emph{trifocal tensor}, as described in \cite[Section 15]{Hartley2004}, gives a polynomial equation that encodes information of when three image lines are the projections of the same world line points; i.e., when they are a point in $\mathcal L_{\mathcal C}$. The trifocal tensor is a $3\times 3\times 3$ tensor (i.e., a bilinear map $\mathbb C^3\times \mathbb C^3\to\mathbb C^3$). It is defined as follows. Let $\ell=(\ell_1,\ell_2,\ell_3)\in\mathcal L_{\mathcal C}$ and $h_i=C_i^T\ell_i$. We know from Theorem~\ref{thm: main1} that $\operatorname{rank}\ M(\ell)\leq 2$. Since $C_1$ has full-rank, we can find an invertible matrix $A\in\mathbb C^{4\times 4}$ such that $Ah_1=(\ell_1,0)^T$, and we have $\operatorname{rank} M(\ell) =\operatorname{rank} A\,M(\ell)$. The upper $3\times 3$ determinant of the $4\times 3$ matrix $AM(\ell)$ vanishes and therefore gives a trilinear equation of the form $\ell_1^T\,T_1(\ell_2,\ell_3)=0$. The bilinear map $T_1$ is the trifocal tensor. Similarly, we can find trifocal tensors $T_2$ and $T_3$.
\end{remark}

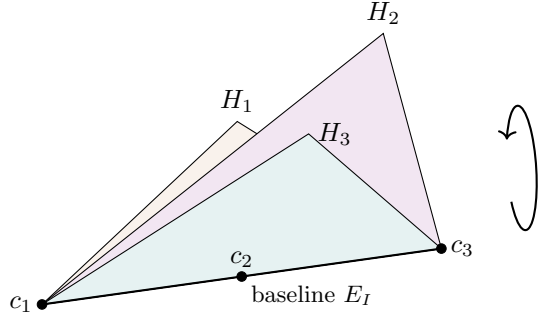
\begin{figure}
\begin{center}
    \begin{tikzpicture}[rotate around x=10, rotate around y=25, rotate around z=0, scale=0.9]
    \coordinate (p) at (-1,0,-2);
    \coordinate (q) at (0,-1,-3);
    \coordinate (r) at (1,2,0);
    \coordinate (c) at (-4,0,2);
    \coordinate (c2) at (-1.25,-0.25,2);
    \coordinate (c3) at (1.5,-0.5,2);
    \draw[fill=brown!10] (c) -- (c3) -- (p) -- cycle;
    \draw[fill=violet!10] (c) -- (c3) -- (r) -- cycle;
    \draw[fill=teal!10] (c) -- (c3) -- (q) -- cycle;
    \draw[thick] (c) -- (c3);
    \draw (p) node[above] {${H_1}$};
    \draw (q) node[right] {${H_3}$};
    \draw (r) node[above] {${H_2}$};
    \draw[fill] (c) circle (2pt) node[left] {${c_1}$};
    \draw[fill] (c2) circle (2pt) node[above] {${c_2}$};
    \draw[fill] (c3) circle (2pt) node[right] {${c_3}$};
    \draw (c2) node[below right] {\small baseline $E_I$};
    \draw[thick, ->] (2.5,0,2) arc (-150:150:0.2 and 1);
    \end{tikzpicture}
\end{center}
\caption{\label{fig4} {The picture shows three collinear cameras $c_1,c_2,c_3$ and three back-projected planes $H_1,H_2,H_3$. Every arrangement of planes that is obtained by rotating $H_1,H_2,H_3$ individually around the baseline gives again three back-projected planes. This shows that for $k$ collinear cameras, the variety $\left\{\ell \in (\mathbb P^2)^m \mid \mathrm{rank}\ M(\ell) \leq 2\right\}$ has a $k$-dimensional irreducible component that is different from $\mathcal L_{\mathcal C}$. Hence, when we have $k\geq 4$ collinear cameras we can't have equality in Theorem \ref{thm: main1}, because $\mathcal L_{\mathcal C}$ is irreducible and of dimension 4 by Theorem \ref{thm_dimension}.}}
\end{figure}

{When there are four or more collinear cameras, we need more constraints -- Figure \ref{fig4} shows why in this case we can't have equality in Theorem \ref{thm: main1}.

We explain the additional polynomial equations when we have four or more collinear cameras. For this let} $I\subset  \{1,\ldots,m\}$ be a subset of indices such that $|I|\geq 4$, and the camera centers with index in $I$ are collinear; that is, such that $E_I$ as defined in~\eqref{def_E_I} is a line. We denote its dual line relative to the standard Hermitian inner product by
$$E_I^*:=\left\{p\in \PP^3\mid p^*q = 0 \text{ for all } q\in E_I\right\}\in\mathbb G.$$
\begin{remark}
The reason why we use the Hermitian inner product here is that $E_I\cap E_I^*= \emptyset$ for any sets of cameras. By contrast, we can have $E_I=E_I^\perp$ (where the latter is defined as in (\ref{def_dual_lines})). For instance, if $E_I$ is spanned by $[1:i:0:0]$ and $[0:0:1:i]$, then $p^Tp=0$ for every point $p\in E_I$. The proof of Theorem \ref{thm: main2} below is based on the assumption that $E_I$ and its dual are two different lines, and this is why we use here the Hermitian dual, not the Euclidean.
\end{remark}

For $\ell=(\ell_1,\ldots,\ell_m)\in (\PP^2)^m$ we denote, as before, the back-projected planes by $H_i$. We write
\begin{equation}\label{def_F_I}
F_{I}(\ell_i): = \operatorname{span}\big(\{c_i\} \cup (H_i\cap E_I^*)\big).
\end{equation}
If $E_I^*$ is not contained in $H_i$, they meet in a point $q$ and $F_I(\ell_i)$ is the line through $c_i$ and $q$. If $E_I^*$ is contained in $H_i$, then $F_I(\ell_i)$ is the plane spanned by $c_i$ and $E_I^*$. Associated to $i$ and $I$ we denote the Schubert variety of lines intersecting $F_{I}(\ell_i)$ by
\begin{equation*}
\Omega_I(\ell_i):=\{L\in\mathbb G \mid L\cap F_{I}(\ell_i) \neq \emptyset\}.
\end{equation*}
Figure~\ref{fig1} provides a geometric interpretation of these Schubert varieties.

Generically, $F_{I}(\ell_i)$ is a line (depicted as the red lines in Figure \ref{fig1}). When $|I|\ge 4= \dim \mathbb G$,  we expect $\bigcap_{i\in I} \Omega_I(\ell_i)$ to be zero-dimensional or empty, because each $\Omega_I(\ell_i)$ is generically a hypersurface. We denote the exceptional locus by~$\mathcal Y_{\mathcal{C},I}:=\left\{\ell=(\ell_1,\ldots,\ell_m)\in (\PP^2)^m\mid \dim \bigcap_{i\in I} \Omega_I(\ell_i) \geq 1\right\}.$
This is an algebraic subvariety of $ (\PP^2)^m$. To see this, recall from Lemma \ref{le: inf_trans} that $\ell\in \mathcal Y_{\mathcal{C},I}$, if and only if there are at least three distinct lines intersecting every $F_I(\ell_i)$. Since $E_I$ and $E_I^*$ are distinct and intersect every~$F_I(\ell_i)$, we have to find a third line. Let $f_1,f_2$ be two fixed points that span the line $E_I^*$. {We claim that $\mathcal{Y}_{\mathcal{C},I}\subset (\PP^2)^m$ is the set of points $\ell\in (\PP^2)^m$ such that there is an $L\in \mathbb{G}$ that intersects each $F_I(\ell_i)$ for $i\in I$ in a point of the form $a_i=s_ic_i+t_if_1+u_if_2$, where $s_i=t_i+u_i$. This is an algebraic variety, because projections {from projective varieties} are closed maps \cite[Proposition 7.16]{Gathmann}. Now we prove the claim: If for each $s_i$, we have $s_i=0$, then $[s_i:t_i:u_i]=[0:1:-1]$ and the linear spaces $F_I(\ell_i)$ meet in a common point $q=f_1-f_2$. In the plane spanned by $q$ and $E_I$ there is a 2-dimensional family of lines intersecting each $F_I(\ell_i)$. If $s_i\neq 0$ for some $i$, then $a_i$ does not meet either of $E_I,E_I^*$, and so there is a third line not equal to $E_I$ or $E_I^*$ intersecting each $F_I(\ell_i)$.}

We define
$$\mathcal Y_\mathcal{C} := \bigcap_{I \subset \{1,\ldots,m\}:\: E_I \text{ is a line}} \mathcal{Y}_{\mathcal{C},I}.$$
Notice that, if $I\subset J$ then $\mathcal Y_{\mathcal{C},J}\subset \mathcal Y_{\mathcal{C},I}$. Furthermore, if $|I|\leq 3$, then $\mathcal Y_{\mathcal{C},I} = (\PP^2)^m$. Therefore, if $\mathcal I$ is the set containing all the maximal sets of indices corresponding to four or more collinear cameras, we have the more finely grained description $\mathcal Y_{\mathcal{C}} := \bigcap_{I \in\mathcal I} \mathcal{Y}_{\mathcal{C},I}$.

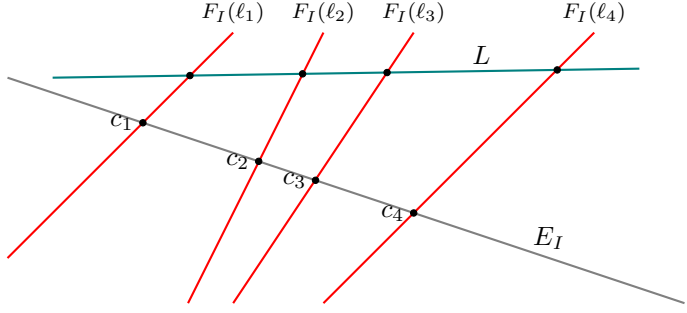
\begin{figure}
    \begin{center}
\begin{tikzpicture}[scale=0.6]
\draw[gray, thick] (-3,1) -- (12,-4);
\draw[red, thick] (-3,-3) -- (2,2);
\draw (2,2) node[anchor=south]{{\footnotesize	$F_I(\ell_1)$}};
\draw[red, thick] (1,-4) -- (4,2); 
\draw (4,2) node[anchor=south]{{\footnotesize	$F_I(\ell_2)$}};
\draw[red, thick] (4,-4) -- (10,2);
\draw (10,2) node[anchor=south]{{\footnotesize	$F_I(\ell_4)$}};
\draw[red, thick] (2,-4) -- (6,2);
\draw (6,2) node[anchor=south]{{\footnotesize	$F_I(\ell_3)$}};
\filldraw[black] (-3+0.2*15,1-0.2*5) circle (2pt);
\filldraw[black] (-3+0.371*15,1-0.371*5) circle (2pt);
\filldraw[black] (-3+0.6*15,1-0.6*5) circle (2pt);
\filldraw[black] (-3+0.455*15,1-0.455*5) circle (2pt);
\draw (-3+0.2*15,1-0.2*5)node[anchor=east]{$c_1$};
\draw (-3+0.371*15,1-0.371*5)node[anchor=east]{$c_2$};
\draw (-3+0.455*15,1-0.455*5)node[anchor=east]{$c_3$};
\draw (-3+0.6*15,1-0.6*5)node[anchor=east]{$c_4$};
\draw (9,-3) node[anchor=south]{$E_I$};
\node[anchor=west] (source) at (-2+0.426*13+0.2,1+0.426*0.2){};
\draw (7.5,1.1) node[anchor=south]{$L$};
\draw[teal, thick] (-2,1) -- (11,1.2);
\filldraw[black] (-2+0.57*13,1+0.57*0.2) circle (2pt);
\filldraw[black] (-2+0.86*13,1+0.86*0.2) circle (2pt);
\filldraw[black] (-2+0.426*13,1+0.426*0.2) circle (2pt);
\filldraw[black] (-2+0.234*13,1+0.234*0.2) circle (2pt);
\end{tikzpicture}
\end{center}
\caption{\label{fig1}The baseline $E_I$ passes through 4 camera centers $c_1,c_2,c_3,c_4$. Generically, a point $\ell=(\ell_1,\ell_2,\ell_3,\ell_4)\in(\PP^2)^4$
defines four lines $F_I(\ell_i)$ (in red) for $1\leq i\leq 4$ in 3-space. By (\ref{intersection_Schubert}), if the 4 red lines are in general position, there are~2 complex lines meeting all of them, and $E_I$ is one of them. In the picture, the green line $L$ is another line meeting all~$F_I(\ell_i)$.
If the locus of lines meeting all 4 red lines is of positive dimension,~$\ell$ belongs to the exceptional locus $\mathcal Y_\mathcal{C}$.}
\end{figure}

The next theorem gives now a full characterization of the line multiview variety in the presence of collinear cameras. We give a proof in Section \ref{proof thm: main2}.

\begin{theorem}\label{thm: main2} Let
$\mathcal{C}$ be a collection of $m$ cameras with distinct centers. 
Then
$$\mathcal{L}_{\mathcal{C}} = \left\{\ell \in (\mathbb P^2)^m \mid \mathrm{rank}\ M(\ell) \leq 2\right\}\cap  \mathcal Y_\mathcal{C}.$$
\end{theorem}
Let us illustrate Theorem \ref{thm: main2} by obtaining explicit equations in $\ell=(\ell_1,\ldots,\ell_4)$ that vanish if and only if $\ell\in\mathcal Y_\mathcal{C}$, in the case of the collinear cameras
$$C_1=\begin{bmatrix}1 & 0 & 0 & 0\\0 & 1 & 0 & 0\\0 & 0 & 1 & 0\end{bmatrix},\quad
C_2=\begin{bmatrix}0 & 1 & 0 & 0\\0 & 0 & 1 & 0\\0 & 0 & 0 & 1\end{bmatrix},\quad
C_3=\begin{bmatrix}1 & 0 & 0 & -1\\0 & 1 & 0 & 0\\0 & 0 & 1 & 0\end{bmatrix},\quad
C_4=\begin{bmatrix}1&  0 & 0 & 1\\0 & 1 & 0 & 0\\0 & 0 & 1 & 0\end{bmatrix}. $$
The camera centers $c_i$ of $C_i$ lie on the baseline $E_I$, $I=\{1,2,3,4\}$, spanned by $c_1=[0:0:0:1]$ and $c_2=[1:0:0:0]$. To determine the multiview variety for this collection of cameras $\mathcal{C}$ we need to calculate the exceptional locus $\mathcal{Y}_{\mathcal{C}}=\mathcal Y_{\mathcal{C},I}$. This variety can be realized in \texttt{Macaulay2} by computing $\mathcal Y_{\mathcal{C},I}$ via variable elimination
as we did above. Recall from~(\ref{def_pluecker}) the Pl\"ucker embedding $\rho:\mathbb G\to \PP(\mathbb C^{4\times 4})$. Let $L\in\mathbb G$ and $L^\perp$ be its dual line as in (\ref{def_dual_lines}). Let us write
$$\rho(L^\perp)=P=\begin{bmatrix}
0 & -p_0 & -p_1 & -p_2\\
p_0 & 0 & -p_3 & -p_4\\
p_1 & p_3 & 0 & -p_5\\
p_2 & p_4 & p_5 & 0
\end{bmatrix}. $$
Thus, $L^\perp$ is represented by Plücker coordinates $p=[p_0:\cdots:p_5]\in \PP^5$. We have that a point $a\in \PP^3$ lies on $L$, if and only if $Pa=0$. We compute two fixed points $f_1,f_2$ that span $E_I^*$. We introduce $8$ variables $t_i,u_i$ for $i=1,\ldots,4$, and set $a_i:=c_i+t_if_1+u_if_2\in  \operatorname{span}(\{c_i\} \cup  E_I^*)$. Adding the equation~$h_i^Ta_i=0$, where $h_i=C_i^T\ell_i$, assures that $L$ intersects $F_I(\ell_i)$ in $a_i$. Moreover, adding a further equation~$1=t_1+u_1$ confirms that we have $a_1\not\in E_I$ and $a_1\not\in E_I^*$, hence $L\neq E_I$ and $L\neq E_I^*$. Therefore we get $\ell\in\mathcal Y_{\mathcal{C},I}$ from the following ideal of polynomials: $\langle Pa_i,h_i^Ta_i\mid i=1,\ldots,4 \rangle +
  \langle t_1+u_1-1\rangle +\langle p^Tp-1, p_0p_5-p_1p_4+p_2p_3 \rangle.
$
Eliminating all variables except $\ell$ we are left with the following principal ideal. For simplicity we write $x=\ell_1,y=\ell_2,z=\ell_3,w=\ell_4$:
$$\langle 2x_3y_2z_2w_2-x_3y_1z_3w_2-x_2y_2z_3w_2 -x_3y_1z_2w_3-x_2y_2z_2w_3+2x_2y_1z_3w_3\rangle.$$
Adding this ideal to the determinantal ideal using \texttt{Macaulay2} we get a prime ideal of dimension~4 as predicted by Theorem \ref{thm: main2}. The code for this example is attached to the \texttt{arXiv} version of this article.

We take this opportunity to highlight some differences compared to the point multiview variety~$\mathcal{M}_{\mathcal C}$. This is the Zariski closure of the image of the map
\begin{equation}\label{def_Phi_C}
\Phi_{\mathcal C}:\PP^3{\dashrightarrow} (\PP^2)^m, \quad x\mapsto (C_1x,\ldots,C_mx).
\end{equation}
As stated in \cite[Lemma 4.1]{agarwal2019ideals}, for any camera arrangement $\mathcal{C}$ of cameras with different centers, we have
$
    \mathcal{M}_{\mathcal C}=\{(x_1,\ldots,x_m)\in (\PP^2)^m \mid \operatorname{rank}A_{\mathcal C}(x)<m+4\},
$
where $A_{\mathcal C}(x)$ is the $3m\times (m+4)$ matrix
$$A_{\mathcal C}(x)=\begin{bmatrix}C_1 & x_1  &\cdots & 0\\
\vdots & \vdots & \ddots & \vdots\\
C_m  & 0 &\cdots & x_m\\\end{bmatrix};$$
that is, $x\in\mathcal M_{\mathcal C}$ if and only if the the maximal minors of the matrix $A_{\mathcal C}(x)$ vanish.   
The geometric interpretation of the point multiview variety is that $x\in \mathcal M_{\mathcal C}$, if and only if their back--projected lines intersect at least in a point. In contrast to our setting, one does not need an equivalence to $\mathcal{Y}_{\mathcal{C}}$.


We now prove the results in this section. First, we prove Theorem \ref{thm_dimension}, Theorem \ref{thm_real_dimension}, Proposition~\ref{prop joint image} and Proposition \ref{properties_of_LMV}. Thereafter, we prove Theorem \ref{thm: main1}, and Theorem \ref{thm: main2}.

\subsection{Proofs of basic results}\label{sec: proof basics}

\begin{proof}[Proof of Theorem \ref{thm_dimension}]
The map $\Upsilon_{\mathcal C}: \mathbb G \dashrightarrow (\mathbb P^2)^m$ is a rational map. By Lemma \ref{image_is_irreducible} the Zariski closure $\mathcal{L}_{\mathcal{C}}$ of its image is irreducible. Moreover, Lemma \ref{rational_dimension} implies that $\dim(\mathcal{L}_{\mathcal{C}})\leq 4$. We show that $\mathcal{L}_{\mathcal{C}}$ has dimension at least~4. 
Let $\pi:\mathcal{L}_{\mathcal{C}}\to \PP^2\times\PP^2$ be the projection onto the first two factors. The pair~$(\ell_1,\ell_2)$ lies in $(\pi \circ \Upsilon_{\mathcal{C}})(\mathbb G)$ if and only if there is a line $L\in \mathbb{G}$ that projects onto $\ell_1$ and $\ell_2$ respectively, and that does not pass through any camera center. The set of $(\ell_1,\ell_2)$ such that both back-projected planes $C_1^T\ell_1$ and $C_2^T\ell_2$ intersect a given camera center is a proper closed set. Therefore~$(\pi \circ \Upsilon_{\mathcal{C}})(\mathbb G)$ is non-empty Zariski open in $\PP^2\times \PP^2$. This means 
that $\pi$ is dominant. By Lemma \ref{rational_dimension}, $\dim \mathcal{L}_{\mathcal{C}}\geq \mathrm{dim}(\PP^2\times \PP^2)=4$.
\end{proof}

\begin{proof}[Proof of Theorem \ref{thm_real_dimension}] Recall from Section \ref{sec: Grassmanians} the definition of the real Grassmanian $\mathbb G_{\mathbb R}$. If the camera matrices are real, the image of $\mathbb G_{\mathbb R}$ under the joint camera map is contained in $\mathcal L_{\mathcal C}^{\mathbb R}$, so that $\overline{\Upsilon_{\mathcal C}(\mathbb G_{\mathbb R})}\subseteq \overline{\mathcal L_{\mathcal C}^{\mathbb R}}\subseteq \mathcal L_{\mathcal C}$. Recall that  $\overline{\mathbb G_{\mathbb R}}=\mathbb G$.
Applying Lemma \ref{le: polyzar} yields $\overline{\Upsilon_{\mathcal C}(\mathbb G_{\mathbb R})} = \overline{\Upsilon_{\mathcal C}(\mathbb G)} = \mathcal L_{\mathcal C}$. Therefore, $\overline{\mathcal L_{\mathcal C}^{\mathbb R}} = \mathcal L_{\mathcal C}$.
This proves the first part of Theorem~\ref{thm_real_dimension}. The second part we observe that in the case of real cameras, $\mathcal L_{\mathcal C}$ is defined by real polynomial equations by Theorem \ref{thm: main2}. These real equations define the real algebraic variety $\mathcal L_{\mathcal C}^{\mathbb R}$, whose complexification is $\mathcal L_{\mathcal C}$. The real dimension of $\mathcal L_{\mathcal C}^{\mathbb R}$ is~4 by \cite[Theorem 4.3]{breiding2021algebraic}. The statement follows then from \cite[Theorem 1]{whitney}.
\end{proof}

\begin{proof}[Proof of Proposition~\ref{prop joint image}]
Let $\ell \in \mathcal L_{\mathcal C}$. Recall that $\Upsilon_{\mathcal C}^{-1}(\ell)$ consists of the lines contained in $H_1\cap\cdots \cap H_m$ meeting no centers.
Assume that the back-projected planes $H_i$ of $\ell$ meet in exactly a line $L$ in $\PP^3$. If $L$ does not pass through any camera center, then $\Upsilon_{\mathcal{C}}(L)=\ell$, so $\ell$ lies in the image.
If $L$ does pass through the camera center $c_i$, the joint camera map $\Upsilon_{\mathcal{C}}$ is not defined at $L$, so there is no line in $\PP^3$ that projects to $\ell$. Further, in this case $\ell\in \mathcal{X}$, since $c_i^Th_j=0$ for each $j$. Finally, 
we consider when the back-projected planes meet in a plane. In such a plane, we can find a line that does not pass through any camera center, and therefore all such points must lie in the image.
\end{proof}

\begin{proof}[Proof of Proposition \ref{properties_of_LMV}] We first prove item 1. Let $L\in\mathbb G$ and
$\ell=(\ell_1,\ldots,\ell_m)=\Upsilon_{\mathcal C}(L).$ The preimage $\Upsilon_{\mathcal C}^{-1}(\ell)$ consists of the lines contained in $H_1\cap\cdots \cap H_m$ meeting no centers. We see that $\Upsilon_{\mathcal C}^{-1}(\ell)=\{L\}$ if and only if $H_1\cap\cdots \cap H_m=L$, meaning $\mathrm{rank}\; M(\ell)=2$. Otherwise $\mathrm{rank}\; M(\ell)=1$ and $\Upsilon_{\mathcal C}^{-1}(\ell)$ consists of all lines in the plane $H_1\cap\cdots\cap H_m$ that intersects none of the finitely many camera centers, hence $\Upsilon_{\mathcal C}^{-1}(\ell)$ has infinitely many elements.

Next, we prove item 2. We assume that $\mathcal{C}'=\{C_1,\ldots,C_k\}$ and that $\pi$ projects onto the first~$k$ factors. If $L$ passes through no camera center among $\mathcal{C}$, then it passes no camera center of $\mathcal{C}'$, implying $\pi(\Upsilon_{\mathcal{C}}(\mathbb G))\subseteq \Upsilon_{\mathcal{C}'}(\mathbb G)$. 
Using Lemma \ref{le: polyzar} and that $\pi$ is a closed map \cite[Proposition 7.16]{Gathmann}, we have 
$$\mathcal{L}_{\mathcal{C}'}=\overline{\Upsilon_{\mathcal{C}'}(\mathbb G)}\supseteq\overline{\pi(\Upsilon_{\mathcal{C}}(\mathbb G))} =\overline{\pi(\mathcal{L}_{\mathcal{C}})}=\pi(\mathcal{L}_{\mathcal{C}}).$$
By Theorem \ref{thm_dimension}, $\mathcal{L}_{\mathcal{C}'}$ and $\mathcal{L}_{\mathcal{C}}$ are irreducible and of dimension 4. On the other hand, $\pi(\mathcal{L}_{\mathcal{C}})$ is also irreducible by Lemma \ref{image_is_irreducible}. Moreover, we have a dominant rational map $\pi: \mathcal{L}_{\mathcal{C}}\dashrightarrow \pi(\mathcal{L}_{\mathcal{C}})$ showing by Lemma \ref{rational_dimension} that $\dim \pi(\mathcal{L}_{\mathcal{C}})\le 4$, and as in the proof of Theorem \ref{thm_dimension} a dominant rational map $\pi(\mathcal{L}_{\mathcal{C}}) \dashrightarrow \PP^2\times \PP^2$ showing by Lemma \ref{rational_dimension} that $\dim \pi(\mathcal{L}_{\mathcal{C}})\ge 4$. Finally, $\pi(\mathcal{L}_{\mathcal{C}})= \mathcal{L}_{\mathcal{C}'}$ by Lemma~\ref{lemma_X_equal_to_Y}. 
\end{proof}

\subsection{Proof of main results}\label{proof thm: main2}
For the proof we first introduce some notation. We write
\begin{equation}\label{def_M_C}
\mathcal V_{\mathcal C} := \{\ell \in (\mathbb P^2)^m \mid \mathrm{rank}\, \begin{bmatrix}C_1^T\ell_1 & \cdots & C_m^T\ell_m\end{bmatrix} \leq 2\}
\end{equation} 
The basic idea of the proof is to use the fact that Zariski closure coincides with the Euclidean closure 
of $\Upsilon_{\mathcal{C}}(\mathbb G)$, written $\overline{\Upsilon_{\mathcal{C}}(\mathbb G)}^E$, as previously explained follows from Chevalley's theorem; see \cite[Theorem 4.19]{michalek2021invitation}. 

The essential idea of the proof of Theorem \ref{thm: main2} is to show two inclusions. First, we take a point $\ell\in\mathcal V_{\mathcal C}\cap \mathcal{Y}_{\mathcal{C}}$ and then create sequences in the image $\Upsilon_{\mathcal{C}}(\mathbb G)$ converging to $\ell$ in the Euclidean topology. For the other inclusion we construct the necessary set of one-dimensional lines in the condition of $\mathcal{Y}_{\mathcal{C}}$.

In the following, we fix a point $\ell=(\ell_1,\ldots,\ell_m)\in(\PP^2)^m$ and, as before, denote the back-projected planes~$H_i$ defined by $h_i:=C_i^T\ell_i$. We say that a sequence of planes converges, if their equations (which are points in~$\PP^3$) converge in the Euclidean topology. As we approach the proof of Theorem \ref{thm: main2}, we need three lemmata.
\begin{lemma}\label{le:im_in_Y} The image of the joint camera map is a subset of both the determinantal variety $\mathcal{V}_{\mathcal{C}}$ and the exceptional locus $\mathcal{Y}_{\mathcal{C}}$. In other words, $\Upsilon_{\mathcal{C}}(\mathbb{G})\subseteq \mathcal{V}_{\mathcal{C}}\cap \mathcal Y_\mathcal{C}$.
\end{lemma}
\begin{proof} We first show that $\Upsilon_{\mathcal{C}}(\mathbb{G})\subseteq\mathcal{V}_{\mathcal{C}}$. This inclusion follows from the fact that $\ell\in (\PP^2)^m$ lies in the image if and only if there is a line $L\in \mathbb G$ with $\ell=\Upsilon_{\mathcal{C}}(L)$. This happens precisely when the back-projected planes of $\ell$ meet in $L$; the kernel of the matrix $M(\ell)^T$ contains two distinct vectors, meaning it has rank at most 2.

To see that $\Upsilon_{\mathcal{C}}(\mathbb{G})\subseteq\mathcal{Y}_{\mathcal{C}}$, note that if no four cameras are collinear, then $\mathcal{Y}_{\mathcal{C},I}=(\PP^2)^m$ for each set of indices~$I$. 
Now fix a maximal set of four or more collinear cameras $I$, we find a non-empty open subset $U\subseteq \mathbb G$ such that $\Upsilon_{\mathcal{C}}(U)\subseteq \mathcal{Y}_{\mathcal{C},I}$. This is enough by Lemma \ref{le: polyzar}. 

Let $U_1\subseteq \mathbb G$ be the open set of lines not meeting any camera center. For $L\in U_1$ let $\ell=\Upsilon_{\mathcal C}(L)$ and consider $F_I(\ell_i)$. Let $U_2\subseteq \mathbb G$ be the Zariski open set, where the $F_I(\ell_i)$ do not intersect. By construction,~$F_I(\ell_i)$ meets both $E_I$ and $E_I^*$. We write $U_3$ for the Zariski open set of lines that intersects neither $E_I$ nor $E_I^*$.
We next argue that a line $L\in U_1\cap U_2\cap U_3$ intersects each $F_I(\ell_i)$ { for $\ell=\Upsilon_{\mathcal{C}}(L)$}. This is because both $L$ and $F_I(\ell_i)$ lie in the plane $H_i$, where $H_i$ is the back-projected plane of $\ell_i$, and so must have an intersection point. Now we have three distinct lines in $\bigcap_{i\in I}\Omega_I(\ell_i)$, namely $E_I,E_I^*$ and~$L$.
According to Lemma~\ref{le: inf_trans}, there are infinitely many lines intersecting each~$F_I(\ell_i)$. This implies that $\bigcap_{i\in I}\Omega_I(\ell_i)$ contains a one-dimensional family of lines. Letting $U=U_1\cap U_2\cap U_3$, we are done.
\end{proof}

\begin{lemma}\label{le: family of lines} 
Let $\ell=(\ell_1,\ldots,\ell_m)\in (\PP^2)^m$. If the back-projected planes $H_1,\ldots,H_m$ of $\ell$ intersect in exactly a line $L$ that goes through only one camera center, then $\ell\in\overline{\Upsilon_{\mathcal{C}}(\mathbb G)}^\mathrm{E}$.
\end{lemma}
\begin{proof}
We can assume without loss of generality that $c_1$ is the unique camera center contained in $L$.
Consider a sequence of lines $L^{(n)}$ in the plane $H_1$ that do not cross any camera center, and that tends toward $L=H_1 \cap \cdots \cap H_m$. Such a sequence exists, because the Schubert variety of lines in $H_1$ meeting (at least) one of the camera centers is closed and of lower dimension. 
For every~$n$, let
$$\ell^{(n)}=(\ell_1^{(n)},\ldots,\ell_m^{(n)}):=\Upsilon_{\mathcal C}(L^{(n)})$$
and let $H_i^{(n)}$ be the back-projected plane of $\ell_i^{(n)}$. We have $L^{(n)}\subset H_1^{(n)}\cap\cdots\cap H_m^{(n)}$. For every~$i$ and every $n$, the plane $H_i^{(n)}$ is spanned by the camera center $c_i$ and~$L^{(n)}$. Since for $i>1$ we have $c_i\not\in L$, this implies that $H_i^{(n)}$ tends to the plane spanned by $c_i$ and $L$, which is precisely $H_i$. Consequently,~$\ell_i^{(n)}\to \ell_i$ for $i>1$.
For $i=1$ we use that the plane $H_1^{(n)}$ is spanned by $c_1$ and $L^{(n)}$, and that the latter is a line contained in~$H_1$.
Therefore, $H_1^{(n)}=H_1$ for every $n$, and hence $\ell_1^{(n)}= \ell_1$, because the map that sends back-projected planes to lines in $\PP^2$ is continuous. This shows that $\ell^{(n)}\to \ell$.
\end{proof}

\begin{figure}
    \begin{center}
\begin{tikzpicture}[scale=0.6]
\draw[gray, thick] (-3,1) -- (12,-4);
\draw[red, thick] (-3,-3) -- (2,2);
\draw[red, thick] (2,-4) -- (6,2);

\draw (9,-3) node[anchor=south]{$E_I$};
\draw (7.5,1.1) node[anchor=south]{$L^{(n)}$};
\draw[teal, thick] (-2,1) -- (11,1.2);

\node (c1) at (-3+0.2*15,1-0.2*5){};
\node (c2) at (-3+0.455*15,1-0.455*5){};
\node (un) at (-2+0.234*13,1+0.234*0.2){};
\node (vn) at (-2+0.57*13,1+0.57*0.2){};

\filldraw[black] (c1) circle (2pt);
\filldraw[black] (c2) circle (2pt);
\draw (c1)node[anchor=south east]{$c_1$};
\draw (c2)node[anchor=north west]{$c_2$};
\filldraw[black] (un) circle (2pt);
\filldraw[black] (vn) circle (2pt);
\draw (un)node[anchor=south east]{$u_n$};
\draw (vn) node[anchor=south east]{$v_n$};

\filldraw[fill=gray!20] (c1.center) -- (c2.center) -- (un.center) -- cycle;
\filldraw[fill=blue!20] (c1.center) -- (vn.center) -- (un.center) -- cycle;
\draw[black, dotted] (c2.center) -- (un.center);

\node[anchor=west] (source) at (0.426*13+0.2,1+0.426*0.2){};
\node (destination) at (-1+0.371*15-0.15,1-0.371*5-0.6){};
\draw[->, dotted, thick](source)--(destination);
\end{tikzpicture}
\end{center}
\caption{\label{fig2} A cartoon of the proof strategy for Theorem \ref{thm: main2}. The green line 
represents $L^{(n)}$, which approaches the black baseline $E_I$. It intersects the first red line $F_I(\ell_1)$ in $u_n$ and the second red line $F_I(\ell_1)$ in $v_n$.
The grey plane $H_1$ is spanned by $c_1,u_n$ and $c_2$, while the blue plane $H_1^{(n)}$ is spanned by $c_1,u_n$ and $v_n$. Since $v_n\to c_2$, we have $H_1^{(n)}\to H_1$, which translates into $\ell_1^{(n)}\to \ell_1$.}
\end{figure}
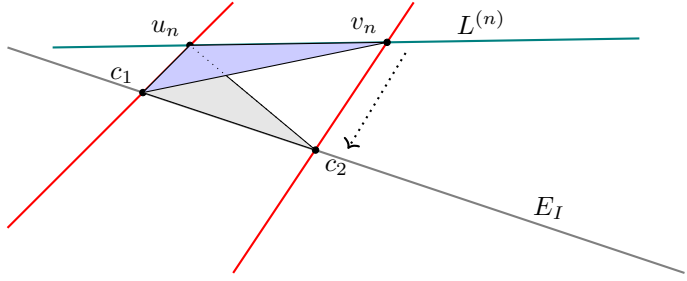

We are now equipped with everything we need to prove Theorem \ref{thm: main2}.

\begin{proof}[Proof of Theorem \ref{thm: main2}]
{We equivalently show that
\begin{equation}\label{eq1: main2}
\overline{\Upsilon_{\mathcal C}(\mathbb G)}^\mathrm{E} = \mathcal V_{\mathcal C}\cap \mathcal Y_\mathcal{C},
\end{equation}
since the Euclidean closure of $\Upsilon_{\mathcal C}(\mathbb G)$ is equal to $\mathcal L_{\mathcal C}$ by Chevalley's theorem as pointed out previously. We first show the inclusion from left to right in (\ref{eq1: main2}). We have $\Upsilon_{\mathcal C}(\mathbb G) \subseteq \mathcal V_{\mathcal C}\cap \mathcal Y_\mathcal{C}$ by Lemma \ref{le:im_in_Y}, which implies $\overline{\Upsilon_{\mathcal C}(\mathbb G)}^\mathrm{E} \subseteq \mathcal V_{\mathcal C}\cap \mathcal Y_\mathcal{C}$, since $  \mathcal V_{\mathcal C}\cap \mathcal Y_\mathcal{C}$ is closed in the Euclidean topology. 

Next, we show the inclusion from right to left in (\ref{eq1: main2}). Take $\ell\in\mathcal V_{\mathcal C}\cap \mathcal Y_\mathcal{C} $, we prove that $\ell \in \overline{\Upsilon_{\mathcal{C}}(\mathbb G)}^\mathrm{E}$. Since $\ell\in \mathcal V_{\mathcal C}$, the back-projected planes $H_i$ of $\ell_i$ must meet in at least a line~$L$. If there is such a line $L$ that contains no camera center, then $\Upsilon_{\mathcal C}$ is defined at~$L$ and we have $\ell=\Upsilon_{\mathcal C}(L)\in \overline{\Upsilon_{\mathcal{C}}(\mathbb G)}^\mathrm{E}$. If there exists a line in the intersection of the back-projected planes that contains exactly one camera center, Lemma \ref{le: family of lines} tells us that $\ell\in \overline{\Upsilon_{\mathcal{C}}(\mathbb G)}^\mathrm{E}$. 

Otherwise, the back-projected planes $H_1,\ldots,H_m$ meet in exactly a line $L$ that contains at least two camera centers. We now use the fact that $\ell\in \mathcal{Y}_\mathcal{C}$ to prove that $\ell \in \overline{\Upsilon_{\mathcal{C}}(\mathbb G)}^\mathrm{E}$. Let $I\subseteq \{1,\ldots,m\}$ be the indices of cameras whose centers lie on~$L$. Note that each $F_I(\ell_i)$ is a line, because $L$ does not meet $E_I^*$. We look at three separate cases. 

\emph{Case 1:} If $|I|=2$, then, assuming without restriction that $I=\{1,2\}$, we can construct a sequence of lines $L^{(n)}=\mathrm{span}\{u_n,v_n\}$, $u_n\in F_I(\ell_1),v_n\in F_I(\ell_2)$ meeting no center such that $u_n\to c_1,v_n\to c_2$. 
Consider the sequence $\ell^{(n)}:=\Upsilon_{\mathcal{C}}(L^{(n)})$ and denote by~$H_i^{(n)}$ the back-projected plane of $\ell_i^{(n)}$. We have to show that~$H_i^{(n)}\to H_i$. The plane $H_i^{(n)}$ is spanned by the camera center $c_i$ and~$L^{(n)}$. Further, for $i\not\in I$ we have $c_i\not\in L$, which implies that $H_i^{(n)}\to \operatorname{span}(\{c_i\}\cup L) = H_i$. Note that the map which takes back-projected planes to lines in $\PP^2$ is continuous. 
Consequently, we have $\ell_i^{(n)}\to \ell_i$ for every $i\not\in I$. It remains to discuss the case $i\in I$. Without restriction we can assume that $i=1$. Note that $u_n,c_1,v_n$ are three distinct points that span $H_1^{(n)}$, because $L^{(n)}$ does not meet $c_1$. On the other hand, $H_1$ is spanned by $u_n,c_1,c_2$ for any $n$. Now observe that the line $F_I(\ell_1)$ is spanned by $c_1,u_n$ for any $n$. Consequently, the plane $H_1^{(n)}$ is also spanned by $u_1,c_1,v_n$, and $H_1$ is spanned by $u_1,c_1,c_2$. Since $v_n\to c_2$ this shows that $H_1^{(n)}\to H_1$, and so $\ell_1^{(n)}\to \ell_1$.  

\emph{Case 2:} $|I|\ge 3$ and three of $F_I(\ell_i),i\in I,$ are disjoint lines. Since $\ell\in \mathcal Y_{\mathcal{C},I}${,} we have by definition of $\mathcal Y_{\mathcal{C},I}$ that there is a one-dimensional family of lines through each $F_I(\ell_i),i\in I$. By Lemma \ref{le: three lines} this family sits in a smooth unique quadric. And Lemma \ref{le: cont_fam_lines} says that all $F_I(\ell_i),i\in I,$ are disjoint and there is a one-dimensional family of lines~$L^{(n)}$ in this quadric continuously approaching $L$.  By Lemma \ref{le: cont_fam_lines}, there are exactly two lines in the smooth quadric meeting a center $c_i$ for~$i\in I$. So by taking a subsequence, we may assume $L^{(n)}$ meets no camera center. We set 
$\ell^{(n)}:=\Upsilon_{\mathcal{C}}(L^{(n)})$. Denote by~$H_i^{(n)}$ the back-projected plane of $\ell_i^{(n)}$. We have to show that~$H_i^{(n)}\to H_i$. As in Case 1, we have $\ell_i^{(n)}\to \ell_i$ for every $i\not\in I$. It remains to discuss the case $i\in I$. Without restriction we can assume that $i=1$ and that $c_1,c_2\in L$. We have $L^{(n)}\in \Omega_I(\ell_1)$ and $L^{(n)}\in \Omega_I(\ell_2)$. Since the $F_I(\ell_i)$ are disjoint lines, none of them is equal to $L^{(n)}$ (for any fixed $n$), and this implies that the line $L^{(n)}$ meets $F_I(\ell_1)$ in a unique point $u_n$ and it meets $F_I(\ell_2)$ in a unique point $v_n$ (depicted in Figure \ref{fig2}). Moreover, we have $u_n\neq v_n$ since the $F_I(\ell_i)$ are disjoint, so that $u_n,c_1,v_n$ are three distinct points that span $H_1^{(n)}$. Analogously to Case 1, $\ell_1^{(n)}\to \ell_1$.


\emph{Case 3:} $|I|\ge 3$ and no three of $F_I(\ell_i),i\in I,$ are disjoint. Then the lines $F_I(\ell_i), i\in I,$ lie in a union $P_1\cup P_2$ of two planes with $L\subset P_1$ and $L\subset P_2$, which we now argue for. Note that either all $F_I(\ell_i)$ are contained in one plane, in which case all $F_I(\ell_i)$ intersect each other, or there are two disjoint $F_I(\ell_i)$, say for indices $i_1,i_2$. Let $P_1=\mathrm{span}\{L,F_I(\ell_{i_1})\}$ and $P_2=\mathrm{span}\{L,F_I(\ell_{i_2})\}$. Now any $F_I(\ell_i)$ lies in either $P_1$ or $P_2$. This is because no three $F_I(\ell_i)$ are disjoint, so $F_I(\ell_i)$ must meet one of $F_I(\ell_{i_1})$ and $F_I(\ell_{i_2})$ (and this intersection is outside $c_i$). 

If $P_1=P_2$, meaning all $F_I(\ell_i),i\in I,$ lie in a plane, then any line in this plane meets each $F_I(\ell_i)$. We can choose a sequence of lines~$L^{(n)}$ in this plane meeting no center, and approaching~$L$. The argument showing that $\ell^{(n)}:=\Upsilon_{\mathcal{C}}(L^{(n)})$ tends to $\ell$ is analogous to Case 1. 

In the case that $P_1\neq P_2$, we first show by contradiction that all lines $F_I(\ell_i), i\in I,$ except for exactly one are contained in the same plane. Suppose that $F_I(\ell_1),F_I(\ell_2)\subset P_1$ and $F_I(\ell_3),F_I(\ell_4)\subset P_2$. Then $F_I(\ell_1),F_I(\ell_2)$ meet in a point $a_1\in P_1$ and $F_I(\ell_3),F_I(\ell_4)$ meet in a point $a_2\in P_2$. Notice that both $a_1,a_2$ lie on $E_I^*$ (for instance $\mathrm{span}\{F_I(\ell_1),F_I(\ell_2)\}$ is a plane containing $E_I$ so it meets $E_I^*$ in exactly a point). Observe that $a_1\neq a_2$, since otherwise $P_1,P_2$ would have $E_I$ in common and an additional point, implying~$P_1=P_2$. Any line distinct from $E_I$ which intersects all the $F_I(\ell_i)$ must then contain both $a_1$ and $a_2$. Consequently, there is only a single such line through both $a_1$ and $a_2$, but this contradicts $\ell\in \mathcal Y_{\mathcal{C},I}$. 

Therefore, without restriction there is exactly one of the $F_I(\ell_i),i\in I,$ contained in the plane $P_1$. After relabeling this line is $F_I(\ell_{1})$. Consider two more indices in $I$, which we can assume to be $2,3\in I$. Then, $F_I(\ell_{2}), F_I(\ell_{3})\subset P_2$. Furthermore, consider three sequences of disjoint lines $G_{i}^{(n)}, i=1,2,3$, that meet $c_{i}$ and~$E_I^*$, and such that $G_{i}^{(n)}\to F_I(\ell_i)$. By Lemma \ref{le: three lines}, for a fixed $n$ the lines $G_{1}^{(n)},G_{2}^{(n)},G_{3}^{(n)}$ determine a smooth quadric $Q^{(n)}$. There is a subsequence of $Q^{(n)}$ that converges because the set of projective quadrics is compact, and for this subsequence $\lim_{n\to\infty} Q^{(n)} = P_1\cup P_2$, for some plane $P_1$ containing $F_I(\ell_1)$ and some plane $P_2$ containing $F_I(\ell_{2}), F_I(\ell_{3})$. Notice that $c_1,c_2,c_3\in Q^{(n)}$ for every $n$, so the whole line through $c_1,c_2,c_3$ is contained in $Q^{(n)}$, which implies that $c_i\in Q^{(n)}$ for all $i\in I$. For the other lines $F_I(\ell_i)$ with $i\in I$ we get sequences
$$G_i^{(n)}:= \text{the unique line in $Q^{(n)}$ through $c_i$ meeting $E_I^*$}.$$
The set of lines through $c_i$ meeting $E_I^*$ is compact so there is a subsequence such that each $G_i^{(n)}$ converges, say to $G_i$. We consider this subsequence. We must have then that the limit $G_i$ lies in either $P_1$ or $P_2$. We show that $G_i= F_I(\ell_i)$ for each $i\in I$. 

For every $n$ there is a one-dimensional family of lines in $Q^{(n)}$, each meeting every $G_i^{(n)}$. This must also hold true in the limit $P_1\cup P_2$. Notice $\lim_{n\to \infty} G_1^{(n)} = F_I(\ell_1)\subseteq  P_1$. If there were another sequence of lines whose limit $G_i$ is in $P_1$, then we can argue as above that there is no one-dimensional family in $P_1\cup  P_2$ meeting every $G_i$. So,  $\lim_{n\to \infty} G_i^{(n)} \subseteq  P_2$ for $i\in I\setminus\{1\}$.
But $P_2$ meets $E_I^*$ in a unique point $a$. Therefore, all $F_I(\ell_i)$ with $i\in I\setminus\{1\}$ meet $E_I^*$ in $a$. We also have $c_i\in F_I(\ell_i)$ by construction. Further, each $G_i^{(n)}$ meets $E_I^*$. Therefore, the limit of $G_i^{(n)}$ also meets $E_I^*$ and it lies in $P_2$, so $\lim_{n\to\infty} G_i^{(n)}$ meets $a$. Then, for $i\in I$ both $\lim_{n\to\infty} G_i^{(n)}$ and $ F_I(\ell_i)$ contain both $a$ and $c_i$, so they are equal. Now, we define the sequence of back-projected planes
$$H_i^{(n)} := \mathrm{span}(L\cup G_i^{(n)})\text{ for }i \in I \quad \text{ and } \quad  H_i^{(n)} := \mathrm{span}(L\cup \{c_i\})\text{ for }i \not\in I.$$
Let $\ell^{(n)}$ correspond to these back-projected planes. Note that $\ell^{(n)}\to \ell$, because $H_i^{(n)}\to H_i$, since we have $G_i^{(n)}\to F_I(\ell_i)$. Finally, $G_1^{(n)}, G_2^{(n)}, G_3^{(n)}$ are disjoint by construction, so we are now in Case 2 and $\ell^{(n)}\in \overline{\Upsilon_{\mathcal{C}}(\mathbb G)}^\mathrm{E}$ for every $n$, which also shows $\ell\in\overline{\Upsilon_{\mathcal{C}}(\mathbb G)}^\mathrm{E}$.}
\end{proof}



\begin{proof}[Proof of Theorem \ref{thm: main1}] By Theorem \ref{thm: main2} we have $\mathcal L_{\mathcal C} = \mathcal V_{\mathcal C}\cap \mathcal Y_\mathcal{C}$. Assume first that no four cameras are collinear. Then for each collection of indices of collinear cameras $I$ we have $\mathcal Y_{\mathcal{C},I} = (\PP^2)^m$. Therefore, $\mathcal Y_\mathcal{C} = \bigcap_{I} \mathcal{Y}_{\mathcal{C},I} = (\PP^2)^m$, which shows one direction. For the other direction, we assume that there exist indices $I=\{1,\ldots,4\}$ of collinear cameras. For general $\ell=(\ell_1,\ldots,\ell_m)\in \mathcal{V}_{\mathcal{C}}$ the~$F_I(\ell_i)$ are disjoint lines. By Lemma \ref{le: three lines}, 
the first three lines $F_I(\ell_1),F_I(\ell_2),F_I(\ell_3)$ lie on a unique smooth quadric~$Q$. Since $Q$ is smooth, it does not contain any planes. Therefore, general points on $\mathrm{span}(\{c_4\}\cap E_I^*)$ do not lie on $Q$. This implies that the line $F_I(\ell_4)$, which is general in $\mathrm{span}(\{c_4\}\cap E_I^*)$, is not contained in $Q$. Lemma \ref{le: nongeneric_schubert} implies that $\bigcap_{i\in I} \Omega(\ell_i)$ is finite. Hence, $\ell\not\in\mathcal Y_\mathcal{C}$.

For the last statement, Theorem \ref{thm_dimension} and Lemma \ref{lemma_X_equal_to_Y} imply that $\mathcal V_{\mathcal C}$ is irreducible and of dimension~4 if and only if $\mathcal L_{\mathcal C}=\mathcal V_{\mathcal C}$, since $\mathcal L_{\mathcal C}\subseteq\mathcal V_{\mathcal C}$ by Theorem \ref{thm: main2}.
\end{proof}

\section{Smoothness} The goal of this section is to prove the following characterization of the smooth locus of the line multiview variety for general cameras.
\begin{theorem}\label{thm:smoothness}
Let $m\geq 3$ and assume no {three} centers are collinear. Then, the singular locus of the line multiview variety is $$\mathcal{L}_{\mathcal{C}}^\mathrm{sing}= \{\ell \in \mathcal{L}_{\mathcal{C}} \mid \mathrm{rank}\ M(\ell) = 1\}.$$
\end{theorem}
Before we prove this theorem, let us state an important consequence.
\begin{corollary}\label{cor_smoothness} Let $\mathcal{C}$ be a collection of $m\ge 3$ cameras, where no {three} centers are collinear.
\begin{enumerate}
\item If $m \geq 4$ and the cameras are not coplanar, $\mathcal{L}_{\mathcal{C}}$ is smooth.
\item If the cameras are coplanar, then $\mathcal{L}_{\mathcal{C}}$ has exactly one singular point, which is the image of any line in the plane spanned by the camera centers. 
\end{enumerate}
\end{corollary}

\begin{proof} By Theorem \ref{thm:smoothness}, the singular locus of $\mathcal{L}_{\mathcal{C}}$ consists of points $\ell\in\mathcal{L}_{\mathcal{C}}$, where $M(\ell)$ has rank one. Recall that this matrix has rank one, if and only if the back-projected planes $H_1,\ldots, H_m$ intersect in a plane, meaning~$H_1=\cdots = H_m$.

For item~1.\ we use that $c_i\in H_i$. Since $H_1=\cdots = H_m$ would imply that the camera centers lie in a common plane, which means that the cameras are coplanar. Hence, $\mathcal{L}_{\mathcal{C}}^\mathrm{sing}=\emptyset$ and $\mathcal{L}_{\mathcal{C}}$ is smooth.
For item~2.\ the only possibility for $H:=H_1=\cdots = H_m$ is the unique plane $H$ where the centers $c_1,\ldots,c_m$ lie. It corresponds to the point $\ell=(\ell_1,\ldots,\ell_m)$, where $\ell_i$ is the image of any line in $H$ not passing through any~$c_i$. 
\end{proof}

Let us compare this result to the case of the point multivariety $\mathcal{M}_{\mathcal C}$.
By \cite[Proposition 4]{trager2015joint}, when the cameras are not collinear, $\mathcal{M}_{\mathcal C}$ is smooth.
When the camera centers $c_i$ are collinear,
then $\mathcal{M}_{\mathcal C}$ has a unique singular point given by the $n$-tuple of
\textit{epipoles} $(\rho_1,\ldots, \rho_m)$, where $\rho_i = C_ic_j$ is image of~$c_j$, $i\neq j$ (since the camera centers lie on a line, all camera centers $c_j,j\neq i,$ project to the same image). In particular, for $m= 2$ the point multiview variety is singular. By contrast, the line multiview variety $\mathcal L_{\mathcal C}$ for $m=2$ is equal to $\PP^2\times \PP^2$ and hence smooth. For $m= 3$ general cameras, $\mathcal L_{\mathcal C}$ has one singular point and for $m\ge 4$ general cameras $\mathcal L_{\mathcal C}$ is smooth by Corollary \ref{cor_smoothness}.

Recall from~(\ref{def_M_C}) the definition of $\mathcal V_{\mathcal C}$ and denote
\begin{equation}\label{def_U_C}
   \mathcal{U}_{\mathcal{C}} = \{h =(h_1,\ldots,h_m) \in (\mathbb P^3)^m\mid \mathrm{rank}\, \begin{bmatrix}h_1 & \cdots & h_m\end{bmatrix}\leq 2 \text{ and } c_i^Th_i=0 \text{ for } 1\leq i\leq m.\}
\end{equation}
So, $\mathcal{U}_{\mathcal{C}}$ is the variety of back-projected planes for points in $\mathcal V_{\mathcal C}$.

\begin{lemma}\label{isomorphism_LMV_and_BPP}
$\mathcal V_{\mathcal C}$ and $\mathcal{U}_{\mathcal C}$ are isomorphic.
\end{lemma}
\begin{proof}
We have a regular map
$$\phi: \mathcal{V}_{\mathcal{C}}\rightarrow\mathcal{U}_{\mathcal{C}}, \; \ell_i\mapsto h_i = C_i^T\ell_i, $$
that is well-defined since any $c_i^Th_i=c_i^TC_i^T\ell_i=(C_ic_i)^T\ell_i=0$. We have a second regular map
$$\psi:\mathcal{U}_{\mathcal{C}}\rightarrow \mathcal{V}_{\mathcal{C}}, \; h_i\mapsto \ell_i = (C_{i}^{T})^{\dagger}h_i,$$
where $(C_{i}^{T})^{\dagger}=(C_iC_i^T)^{-1}C_i$ is the pseudo-inverse of the full rank matrix $C_i^T$. It has the property that $(C_{i}^{T})^{\dagger}C_{i}^{T}=\mathrm{Id}_{\mathbb C^3}$, which shows that $\psi\circ \phi = \mathrm{Id}_{\mathcal V_{\mathcal C}}$. Furthermore, $C_i^T(C_i^T)^\dagger = C_i^T(C_iC_i^T)^{-1}C_i$ is the matrix representation of the projection from $\PP^3$ onto the column span of $C_i^T$, which implies that $\phi\circ \psi = \mathrm{Id}_{\mathcal{U}_{\mathcal C}}$. Hence, $\phi$ and $\psi$ are inverses of each other.
\end{proof}
\begin{remark}
The recent result \cite[Lemma 6.3]{Fulvio21Chiara} by Gesmundo and Meroni implies that for generic $c_i$ the variety $\mathcal{U}_{\mathcal{C}}$ is irreducible and of dimension 4. Theorem \ref{thm: main1} together with Lemma \ref{isomorphism_LMV_and_BPP} reveal what generic means in this case. Namely, that no four $c_i$ are collinear.
\end{remark}

In the following, denote by $A_i\subseteq \mathbb G$ the set of lines through the camera center $c_i$.

\begin{lemma}\label{le: no4sing} When no {three} centers are collinear, 
$$\mathcal{L}_{\mathcal{C}}^\mathrm{sing}\subseteq \{\ell \in \mathcal{L}_{\mathcal{C}} \mid \mathrm{rank}\ M(\ell) = 1\}.$$
\end{lemma}
\begin{proof}
{We denote $A(h):=\begin{bmatrix}h_1 & \cdots & h_m\end{bmatrix}$.

By Theorem \ref{thm: main1}, we have $\mathcal L_{\mathcal C} = \mathcal V_{\mathcal C}$, and by Lemma \ref{isomorphism_LMV_and_BPP} the varieties $\mathcal V_C$ and $\mathcal U_{\mathcal C}$ are isomorphic. Therefore, it suffices to show that points $h\in \mathcal U_{\mathcal C}$, where $\mathrm{rank}\, A(h)=2$, are smooth points of $\mathcal U_{\mathcal C}$. In the following, we fix such a point $h=(h_1,\ldots,h_m)$. 

We introduce the nondegenerate bilinear form on matrices
$$
\langle B_1,B_2\rangle = \mathrm{Trace}(B_1^TB_2).$$
Let $m_1,\ldots,m_N$, $N=4\tbinom{m}{3}$, be the minors of size 3 of the $4\times m$ matrix $A(h)$,  and observe that $c_i^Th_i=\langle A(h), c_ie_i^T\rangle$, where $e_i$ is the $i$th standard basis vector of $\CC^m$. Let 
$$J =\begin{bmatrix} J_1 &  J_2\end{bmatrix}^T, \quad\text{where } J_1 = \begin{bmatrix}
    \frac{\partial m_1}{\partial h} & \ldots &\frac{\partial m_N}{\partial h}
    \end{bmatrix} \text{ and }
    J_2 = \begin{bmatrix}
    \mathrm{vec}(c_1e_1^T) & \ldots &\mathrm{vec}(c_me_m^T)
    \end{bmatrix}
$$
(here, $\mathrm{vec}(\ \cdot\ )$ denotes the vectorization of a matrix).
We show that $h$ is a smooth point by proving that the Jacobian matrix $J\in \mathbb C^{(N+m)\times (4m)}$ at $h$ has rank equal to $\mathrm{codim} (\mathcal{U}_{\mathcal C}) = 3m-4$ (the dimension of $\mathcal{U}_{\mathcal C}$ is 4 by Lemma \ref{isomorphism_LMV_and_BPP}).
This is enough, even if we don't know whether or not the polynomials above generate the ideal of $\mathcal U_{\mathcal C}$. 

We denote the algebraic variety $\mathcal R:= \{(h_1,\ldots,h_m) \in (\mathbb P^3)^m\mid \mathrm{rank}\, A(h)\leq 2\}.
$
We also denote its cone by~$\hat{\mathcal R}:= \{A \in \mathbb C^{4\times m}\mid \mathrm{rank}\, A\leq 2\}$, which is the variety of rank (at most) 2 matrices in $\mathbb C^{4\times m}$. The dimension of the variety $\hat{\mathcal R}$ is $2(4+m-2)=2m+4$, so that $\operatorname{codim} \mathcal R = \operatorname{codim} \hat{\mathcal R} = 2m-4$. The smooth locus of $\mathcal R$ are the matrices of rank exactly two. We have that $I(\mathcal R) = \langle m_1,\ldots,m_N\rangle$, so if $A(h)$ has rank two, $h$ is a smooth point on $\mathcal R$, which implies that 
$\mathrm{rank}\, J_1 = \operatorname{codim} \mathcal R =  2m-4.$
We also have
$\mathrm{rank}\, J_2 = m.$
To show that $J$ has rank $3m-4$, we have to show that the column spans of $J_1$ and $J_2$ intersect trivially.

In the following, we write
$$A:=A(h)\in\mathbb C^{4\times m}$$
Since the rank of $A$ is 2,  we can find rank-$2$ matrices $U\in\mathbb C^{4\times 2}, V\in\mathbb C^{m\times 2}$ such that we have $$A=UV^T.$$
Because $h$ is a smooth point on $\hat{\mathcal R}$, the tangent space of $T_A\hat{\mathcal R}$ consists of derivatives of smooth curves in $\hat{\mathcal R}$ through $A$.
For every $\dot U\in \mathbb C^{4\times 2}$ and $\dot V\in \mathbb C^{m\times 2}$ we have a smooth curve $\gamma(t):=(U+t\dot U)(V+t\dot V)^T\in \mathcal R$. By linearity, we have $\tfrac{\mathrm d }{\mathrm d t}\gamma(t)|_{t=0} = \dot UV^T + U\dot V^T$. This shows
\begin{equation}\label{element_in_TS}
\text{for every } \dot U\in \mathbb C^{4\times 2}, \dot V\in \mathbb C^{m\times 2}: \quad  \dot UV^T + U\dot V^T \in T_A\hat{\mathcal R}.
\end{equation}
Then, the column span of $J_1$ is given by $\{\mathrm{vec}(B) \mid \langle B,X\rangle = 0 \text{ for all } X \in T_A\hat{\mathcal R}\}$.

Take now $B:=\sum_{i=1}^m
\lambda_i\ c_ie_i^T$ and suppose that $B\neq 0$; i.e., $\mathrm{vec}(B)$ is in the column span of $J_2$.
If $\mathrm{vec}(B)$ is also in the column span of $J_1$, then we would have $\langle B, X\rangle = 0$ for every $X\in T_{A}\hat{\mathcal R}$. We find an element in~$T_{A}\hat{\mathcal R}$, where this is not so. By (\ref{element_in_TS}), we can choose $X=\dot UV^T + U\dot V^T$ with $ \dot U\in \mathbb C^{4\times 2}, \dot V\in \mathbb C^{m\times 2}$. Then,
$$\langle B, X\rangle = \langle B, \dot UV^T + U\dot V^T\rangle = \langle B, \dot UV^T\rangle + \langle B, U\dot V^T\rangle = \langle BV, \dot U\rangle + \langle U^TB, \dot V^T\rangle.$$
Without restriction, we can assume that $\lambda_1\neq 0$. Let $L$ be the unique line in the intersection of the back-projected planes defined by $h=(h_1,\ldots,h_m)$; i.e., $L=\{p\in\mathbb P^3\mid h_1^Tp=\cdots =h_m^Tp=0\}$ spans the left kernel of $A$. 

If at least three of the $\lambda_i$ are non-zero, then there exists $\lambda_j\neq 0$, such that $c_j\not\in L$, because at most two camera centers are collinear. In this case, we choose $X\in T_{A}\hat{\mathcal R}$ by taking $\dot U=0$ and $\dot V = e_jx^T$ with $x=V^TA^Tc_j\in\mathbb R^2$. Then,
$$\langle B, X\rangle = \langle U^TB, \dot V^T\rangle=x^TU^TBe_j = 
\lambda_j\ (x^TU^Tc_j) = \lambda_1 (A^Tc_j)^T(A^Tc_j).$$
Recall that $L$ spans the left kernel of $A$. Since $c_j\not\in L$, we have $A^Tc_j\neq 0$, so $\lambda_1 (A^Tc_j)^T(A^Tc_j)\neq 0$. 

The only case that remains is when $\lambda_i=0$ for $i\neq 1,2$ and $c_1,c_2\in L$ (after relabeling). Then, 
$$BV = \lambda_1 c_1 v_1^T + \lambda_2 c_2 v_2^T,\quad \text{ where } v_i:=V^Te_i.$$
We show that $BV\neq 0$.
We have $v_1\neq 0$, because otherwise $h_1=Ae_1 = Uv_1=0$. Similarly, $v_2\neq 0$. So, there exists $w\in \mathbb C^2$ with $v_1^Tw, v_2^Tw\neq 0$. Since $c_1$ and $c_2$ are distinct (and hence linearly independent) and $\lambda_1,\lambda_2$ are not both zero, this gives $BVw = \lambda_1 (v_1^Tw)\ c_1  + \lambda_2 (v_2^Tw)\ c_2\neq 0$. So, $BV\neq 0$. We choose $X\in T_{A}\hat{\mathcal R}$ by setting $\dot U=BV$ and $\dot V=0$. Then, 
$$\langle B, X\rangle = \langle BV, \dot U\rangle = \langle BV, BV\rangle \neq 0.$$

In both cases, there exists $X\in T_{A}\hat{\mathcal R}$ with $\langle B,X\rangle \neq 0$. We have shown that the column spans of $J_1$ and $J_2$ intersect trivially.}
\end{proof}

{In our application of van der Waerden's theorem \ref{purity_thm}, we will develop a birational map $\varphi: X\to \mathcal U_{\mathcal C}$, where $X$ is the \emph{blow-up} of $\mathbb G$ as constructed in Lemma \ref{le: blowup} below. Denote by $A_i\subseteq \mathbb G$ the set of lines through the camera center $c_i$. Every $A_i$ is isomorphic to $\mathbb P^2$, hence smooth. 
\begin{lemma}\label{le: blowup} Consider the blow-up
$$X=\overline{\{(L,\ \mathrm{span}(\{c_1\}\cup L),\ \ldots,\ \mathrm{span}(\{c_m\}\cup L)) \mid L\not\in A_1\cup\cdots\cup A_k\}}\subseteq \mathbb G\times (\mathbb P^3)^{m},$$
where, as before, we identify a plane in $\mathbb P^3$ by its linear equation (a point in $\mathbb P ^3$). Then, the fibers of the projection $\pi$ from $X\subseteq \mathbb G\times \mathcal{U}_{\mathcal{C}}$ to $\mathcal{U}_{\mathcal{C}}$ are singeltons if the planes identified with $(h_1,\ldots,h_m)\in \mathcal{U}_{\mathcal{C}}$ meet in exactly a line and $2$-dimensional otherwise.
\end{lemma}

\begin{proof} We first observe that $\pi$ is surjective by definition. Let $(h_1,\ldots,h_m)\in\mathcal U_{\mathcal C}$. If $(L,h_1,\ldots,h_m)\in X$, then $L\subseteq H_i = \{p\in\mathbb P^3 \mid h_i^Tp=0\}$. We conclude that if $H_1,\ldots,H_m$ meet in a line $L$, then the fiber is exactly the point $(L,h_1,\ldots,h_m)$. If $(H_1,\ldots,H_m)$ meet in a plane, they are all equal: $H:=H_1=\cdots=H_m$. Then, an open dense subset of lines in $H$ meets no centers, and therefore the fiber is the set of points $(L,h_1,\ldots,h_m)$ for any $L\subseteq H$. The variety of lines in $H$ has dimension 2.
\end{proof}

We can now prove Theorem \ref{thm:smoothness}.
\begin{proof}[Proof of Theorem \ref{thm:smoothness}] 
It follows from Lemma \ref{le: no4sing} that $\mathcal L_{\mathcal{C}}^\mathrm{sing}\subseteq \{\ell \in \mathcal{L}_{\mathcal{C}} \mid \mathrm{rank}\ M(\ell) = 1\}$.
If the camera centers are not coplanar, then the right hand side of this is empty, so they are equal.

To complete the proof we now suppose that the camera centers are coplanar. Recall from (\ref{def_U_C}) the definition of the variety $\mathcal U_{\mathcal C}$ of back-projected planes.
 Since by assumption no four camera centers are coplanar, $\mathcal L_{\mathcal{C}}$ is isomorphic to $\mathcal U_{\mathcal C}$ by Theorem \ref{thm: main1} and Lemma \ref{isomorphism_LMV_and_BPP}. We show the equivalent statement that $\mathcal U_{\mathcal{C}}^\mathrm{sing}$ consists of those points $h\in \mathcal U_{\mathcal{C}}$, where $\mathrm{rank}\, \begin{bmatrix}h_1 & \cdots & h_m\end{bmatrix}=1$.

For this, let $X\subseteq \mathbb G\times (\mathbb P^3)^{k}$ be the blow-up as defined as in Lemma \ref{le: blowup}. Consider the projection morphism $\pi:X\to \mathcal{U}_{\mathcal C}$. Via the Segre embedding, we may assume that $\pi$ is a morphism of projective complex spaces (instead of products of projective complex spaces). Let $V$ be an open set in $\mathcal{U}_{\mathcal C}$. Then $\pi^{-1}(V)\to V$ is an isomorphism if and only if $\pi$ is injective on $\pi^{-1}(V)$. We have shown in Lemma \ref{le: blowup} that $\pi^{-1}(h)$ is a singleton, if and only if the planes $H_i=\{h_i=0\}$ intersect in exactly a line. Let
\begin{equation}\label{def_W}
W:=\{h=(h_1,\ldots,h_m)\in \mathcal U_{\mathcal C}\mid H_1,\ldots,H_m \text{ meet exactly in a line}\}.
\end{equation}
We apply van der Waerden's purity theorem (Theorem \ref{purity_thm}) to the birational map $\pi:X\to \mathcal U_{\mathcal C}$. The open set~$W$ in (\ref{def_W}) satisfies the assumptions of Theorem \ref{purity_thm}. By Theorem \ref{thm_dimension}, $\dim \mathcal U_{\mathcal C} = \dim \mathcal L_{\mathcal C} = 4$. By Lemma~\ref{le: blowup}, the fibers of $h\in \mathcal U_{\mathcal C}\setminus W$ are $2$-dimensional. Hence, $X\setminus \pi^{-1}(W)$ has codimension $4-2=2>1$, and so Theorem~\ref{purity_thm} implies that $\mathcal U_{\mathcal C}$ is not smooth; i.e., $\mathcal U_{\mathcal{C}}^\mathrm{sing}\neq \emptyset$.

We have shown that there exists a singular point $h\in \mathcal{U}_{\mathcal{C}}^\mathrm{sing}$ with $\mathrm{rank}\, \begin{bmatrix}h_1 & \cdots & h_m\end{bmatrix}=1$. By assumption that $m\ge 3$ and no three center are collinear, there is a unique such point $h=(h_1,\ldots,h_m)$ corresponding to $H_1=\cdots=H_m$. In other words $W$ is smooth and $\mathcal U_{\mathcal C}\setminus W$ consists of one point. We conclude that this point is the only singular point of $\mathcal{L}_{\mathcal{C}}$.
 \end{proof}}


\section{Multidegrees} The multidegree of the line multiview variety $\mathcal{L}_{\mathcal{C}}$ is defined as the function
$$D(d_{1},\dots,d_{m}):= \#( \mathcal{L}_{\mathcal{C}}\cap (L_{d_{1}}^{(1)}\times \cdots\times L_{d_{m}}^{(m)})), $$
for $(d_{1},\dots,d_{m})\in \mathbb N^{n}$ such that $d_{1}+\cdots +d_{m} = \operatorname{dim} \mathcal{L}_{\mathcal{C}}=4$,  where for each $1\leq i\leq m$ we denote by~$L_{d}^{(i)}\subset \mathbb P^{2}$  a general linear subspace of codimension $d$. The multidegree of a variety in $(\PP^2)^m$
gives its class in the \emph{Chow ring} of $(\PP^2)^m$; see \cite[Chapter~1]{eisenbud-harris:16}. While this is the algebraic interpretation, below we will interpret the multidegree of the line multiview variety from the point of view of computer vision.

We consider a collection $\mathcal C=(C_1,\ldots,C_m)$ of $m$ \emph{general} cameras. This means we take $\mathcal C$ from a Zariski dense subset of all camera tuples, where in particular no four cameras are collinear. Theorem~\ref{thm: main1} implies that for general cameras $\mathcal{L}_{\mathcal{C}}= \{(\ell_{1},\dots,\ell_{m})\in (\mathbb P^2)^{m} \mid \operatorname{rank} M(\ell)\leq 2\}. $ Other than in Theorem~\ref{thm: main1}, here we do not specify the notion of being general in detail.
When the $C_i$ are general, the function $D$ is symmetric meaning that $D(d_{1},\dots,d_{m})=D(d_{\sigma{(1)}},\dots,d_{\sigma{(m)}})$  for any permutation $\sigma$ on $m$ elements. This implies that the multidegree~$D$ is completely determined by the three values $D(2,2,0,\dots,0)$ and $D(2,1,1,0,\dots,0)$ and~$D(1,1,1,1,0,\dots,0)$. We compute them next.

\begin{theorem}\label{thm_mdeg}
For general cameras the multidegree of the line multiview variety $\mathcal{L}_{\mathcal{C}}$
is given by the values $D(2,2,0,\dots,0)=1$ and $D(2,1,1,0,\dots,0)=1$ and $D(1,1,1,1,0,\dots,0)=2$ up to permutation.
\end{theorem}
\begin{proof} 
If the camera matrices are general, no four of the centers are collinear. So by Theorem~\ref{thm: main1} we have $\mathcal{L}_{\mathcal{C}}= \{\ell \in (\mathbb P^2)^m \mid \mathrm{rank}\ M(\ell) \leq 2\}$. As before, we denote $h_i:=C_i^T\ell_i$ and we denote the back-projected planes by~$H_i=\{p\in\PP^3 \mid h_i^Tp=0\}$.
The proof is based on the observation that $\mathrm{rank}\ M(\ell) \leq 2$, if and only if the back-projected planes $H_i$ meet in a line $L$. Such a line uniquely determines $\ell$ by $\ell=\Upsilon_{\mathcal C}(L)$. So, instead of counting $\ell$, we can count the possibilities for $L$.

For $D(2,2,0,\dots,0)$ the first two entries $\ell_1$ and $\ell_2$ are general and fixed. Hence, $H_1$ and $H_2$ are fixed and general. Then $H_1\cap H_2$ meet in exactly a line, which must be $L$. Generically, $L$ does not meet any camera center. Therefore $\Upsilon_{\mathcal{C}}(L)$ is well-defined and determines $\ell_3,\ldots,\ell_m$ uniquely, so that we have $D(2,2,0,\dots,0)=1$.

For $D(2,1,1,0,\dots,0)$, the first entry $\ell_1$ is again general and fixed, which implies that $H_1$ is general and fixed. Furthermore, for general fixed $x_2,x_3\in \PP^2$ we have $x_2^T\ell_2=0,x_3^T\ell_3=0$. Let $K_2$, respectively $K_3$ denote the back-projected line of $x_2$, respectively $x_3$. Then~$H_2$, respectively $H_3$, contains the general line $K_2$, respectively $K_3$, in $\PP^3$. Denote by $q$, respectively $q'$, the unique intersection point of $H_1\cap K_2$, respectively $H_1\cap K_3$. Let $L$ denote the line spanned by $q,q'$; it is the only line in $\PP^3$ that is projected onto $\ell_1,\ell_2$ and $\ell_3$ by the camera matrices $C_1,C_2$ and $C_3$. The line $L$ determines all other~$\ell_i$, meaning $D(2,1,1,0,\dots,0)=1$.

Finally, let us consider  $D(1,1,1,1,0,\dots,0)$. In this case, $H_i$ is constrained to contain a general line $K_i$ in~$\mathbb P^3$ for $1\leq i\leq 4$. By (\ref{intersection_Schubert}), there are two lines meeting four general lines in~$\mathbb P^3$, so that $D(1,1,1,1,0,\dots,0)=2$. The remaining back-projected planes are again uniquely determined after choosing one of the two lines.
\end{proof}

\begin{remark}
In the point multiview variety, the multidegree can be similarly calculated: $\mathcal{M}_{\mathcal{C}}\subseteq (\PP^2)^m$ from Section~\ref{s: line_mult} is of dimension $3$ and for general cameras we therefore need to determine the values of~$D(2,1,0,\ldots,0)$ and $D(1,1,1,0\ldots,0)$. We write $x=(x_1,\ldots,x_m)$ for a point $x\in \mathcal{M}_{\mathcal{C}}$.

To determine $D(2,1,0,\ldots,0)$, we fix generic $x_1$ and let $x_2$ lie in a fixed generic line $\ell$ in $\PP^2$. The back-projected line of $x_1$ and the back-projected plane of $\ell$ generically meet in just one point $X\in \PP^3$, which determines all other components $x_i$, meaning $D(2,1,0,\ldots,0)=1$.
In the case of~$D(1,1,1,0,\ldots,0)$, the three points $x_1,x_2,x_3$ lie on fixed generic lines $\ell_1,\ell_2,\ell_3$ instead. Their back-projected planes meet generically in one unique point $X\in \PP^3$, again showing $D(1,1,1,0,\ldots,0)=1$.

{In the recent work of \cite{escobar2017multidegree}, the multidegree of the \textit{concurrent lines variety}, the variety of lines in $\PP^3$ meeting in a point, was computed. The analogous problem in the line case would be to compute the multidegree of the variety of planes in $\PP^3$ meeting in a line.}
\end{remark}

{Let us discuss Theorem \ref{thm_mdeg} from the point of view of computer vision. Recall that for the line multiview variety we use dual coordinates $\ell\in\mathbb P^2$ which define lines by the equations $x^T\ell = 0$. Putting one linear equation on $\ell$ corresponds to restricting $\ell$ to go through a fixed point in $\mathbb P^2$.

The equation $D(2,2,0,\dots,0)=1$ means that for a general set of $m$ cameras it is enough to take only~2 images of a general line in $\mathbb P^3$ to completely determine the other $m-2$ images.
If the cameras are real, since complex solutions must come in pairs of complex conjugates, we must get $m$ real images.
Furthermore, $D(2,1,1,0,\dots,0)=1$ implies that it is enough to take~1 image of a general line $L$ and to take 2 images of points lying on $L$ to determine the other $m-3$ images. As before, if the cameras are real, we must have $m$ real images. Finally, $D(1,1,1,1,0,\dots,0)=2$ shows that 4 images of 4 points on a general line $L$ in $\PP^3$ determine exactly two $m$-tuples of lines in $\PP^2$. If the cameras are real, these are either both not real or both real. Proposition~\ref{prop_expected} below discusses how many 
real images we can expect when the camera matrices are random real matrices filled with i.i.d.\ standard Gaussian random variables.}
\begin{proposition}\label{prop_expected}
Suppose that the camera matrices $C_1,\ldots,C_m\in\mathbb R^{3\times 4}$ are independent random matrices with i.i.d.\ standard Gaussian entries. For each $1\leq i\leq m$ let $L_{d_i}^{(i)}\subset\mathbb P^2$ be a fixed real linear space of codimension $d_i$, such that four of the $d_i$ are equal to 1 and the rest are zero. Then,  the expected number of real solutions is
$$\operatorname{\mathbb E}  \#( \mathcal{L}_{\mathcal{C}}^{\mathbb R}\cap L_{d_{1}}^{(1)}\times \cdots\times L_{d_{m}}^{(m)}) \approx  1.7262.
 $$
(this means that these are the first digits of the actual value).
\end{proposition}
\begin{remark}On \texttt{MathOverflow}\footnote{\scriptsize\url{https://mathoverflow.net/questions/260607/expected-number-of-lines-meeting-four-given-lines-or-what-is-1-72}} Firsching expanded the number of digits to
$$1.7262312489219034885256331685361697650475579915479447.$$
\end{remark}
\begin{proof}[Proof of Proposition \ref{prop_expected}]
By symmetry, without restriction we can assume that $d_1=d_2=d_3=d_4=1$. Following the arguments in the last paragraph in the proof of Theorem \ref{thm_mdeg} we see that the number of real points in the intersection~$\#( \mathcal{L}_{\mathcal{C}}^{\mathbb R}\cap L_{d_{1}}^{(1)}\times \cdots\times L_{d_{m}}^{(m)})$ is equal to the number of real lines intersecting the four given lines $\{C_i^T\ell_i \mid \ell_i\in L_{1}^{(i)}\}\subset \mathbb P^3$, $1\leq i\leq 4$. These are four independent random elements in the real Grassmannian $\mathbb G_{\mathbb R}$.
Since for any orthogonal matrix $U\in O(4)$ we have that $UC_i^T$ has the same distribution as $C_i^T$, the distribution of the four random lines is invariant under the $O(4)$-action on $\mathbb G_{\mathbb R}$. There is a unique orthogonally invariant probability distribution on the real Grassmannian. With respect to this distribution, B\"urgisser and Lerario showed \cite{PSC} that the first five digits of the expected number of lines intersecting four random independent lines is are 1.7262. The true value of this expected value is only known in the form of an iterated integral; see \cite[Proposition~6.7]{PSC}.
\end{proof}

\section{Euclidean Distance Degree}\label{sec:EDD}
Minimizing the Euclidean distance of a point $u\in\mathbb R^{N}$ to an algebraic variety $X\subset \mathbb R^{N}$ is a fundamental problem in optimization. The first order optimality condition for a smooth point $x\in X$ of this optimization problem is $(x-u)^Tv = 0$ for all $v\in T_xX$, where $T_xX$ denotes the tangent space of~$X$ at~$x$. The \emph{Euclidean Distance Degree} (EDD) \cite{draisma2016euclidean} is motivated by the desire to count the number of points that satisfy these conditions. To get a well-defined count one passes to complex numbers. We consider a point $u\in\mathbb C^{N}$ and an algebraic variety $X\subset \mathbb C^{N}$ and say that a smooth point $x\in X$ is an \emph{ED-critical point}, if $(x-u)^Tv = 0$ for all $v\in T_xX$. The EDD is defined as the number of ED-critical points on $X$ when $u$ is a general point outside $X$. The EDD can be considered as a measure of complexity for solving the optimization problem of minimizing the Euclidean distance from $X$ to~$u$. In this sense, the EDD is important for applied work when data~$u$ comes with noise. To analyze this data one often tries to find the point in the variety (the mathematical model), which is closest to~$u$. It's also important to understand the singular locus when considering the EDD; the actual closest point might be singular and therefore not found as an ED-critical point.

Let us first consider the point multiview variety. Its elements are $m$-tuples of image points. The data structure for images usually is a matrix, where the $(i,j)$ entry stores the information for the pixel with spatial coordinates $i$ and $j$. Therefore, it is meaningful to consider the EDD of the intersection of the point multiview variety (which is a subvariety of $(\PP^2)^m$) with an affine patch. It was shown in \cite{EDDegree_point} that when $m\ge 2$ cameras $\mathcal C=(C_1,\ldots,C_m)$ are in general position, the EDD of  $\mathcal M_{\mathcal C}$ intersected with an affine patch is
\begin{equation}\label{EDD_points}
\frac{9}{2}m^3-\frac{21}{2}m^2+8m-4.
\end{equation}

By contrast, for the line multiview variety, there is no canonical choice of the affine patch. Therefore, we think that Euclidean distance minimization in an affine patch is less meaningful than minimization relative to other distance measures.
One option is to use a distance in the \emph{affine Grassmannian} \cite{affine_G}, which is the space of lines in $\mathbb C^2$. This would take into account the above arguments that image points are usually given in affine coordinates. Alternatively, we can use the distance $d$ from~(\ref{def_metric_Pn_product}), which measures the angle between two linear equations. For this distance, lines are considered close when their equations are close to being linearly dependent. Both models are legitimate. In the following, we discuss the EDD for the angular distance $d$. {In fact, the definition of $d([u],[v])$ in (\ref{def_metric_Pn}) can be expressed as an algebraic function in the homogeneous
coordinates of $u, v$. This already shows that the ED minimization problem is algebraic. We show that is closely connected to the usual EDD of the cone over $\mathcal L_{\mathcal{C}}$.}

Let $\pi:(\mathbb C^{3}\setminus \{0\})^m \to (\PP^2)^m$ be the canonical projection. The cone over the real line multiview variety is $\widehat{\mathcal L}_{\mathcal C}^{\mathbb R}:=\pi^{-1}(\mathcal L_{\mathcal C}^{\mathbb R})\cup\{0\}$. Let $u=(u_1,\ldots,u_m)\in (\PP^2)^m$ and $\hat u =(\hat u_1,\ldots,\hat u_m)\in \pi^{-1}(u)$
be real. Then,
$$\min_{\ell=(\ell_1,\ldots,\ell_m)\in\mathcal L_{\mathcal C}^{\mathbb R}} \,d(\ell,u)^2 =\min_{(\ell_1,\ldots,\ell_m)\in\mathcal L_{\mathcal C}^{\mathbb R}}\, \sum_{i=1}^{m} \min_{t_i\in\mathbb R}\, \frac{\Vert t_i \hat \ell_i - \hat u_i\Vert^2}{\Vert \hat u_i\Vert^2},$$
where $\pi(\hat \ell_1,\ldots,\hat \ell_m) = \ell$ and where $\Vert \cdot \Vert$ is the Euclidean norm. Therefore, if we choose the point $\hat u$ such that $\lambda:=\Vert \hat u_1\Vert = \cdots =\Vert \hat u_m\Vert$, then
\begin{equation}\label{optimization_problem}\min_{\ell\in\mathcal L_{\mathcal C}^{\mathbb R}} \,d(\ell,u)^2
=\lambda^{-1}\min_{\hat\ell\in\widehat{\mathcal L}_{\mathcal C}^{\mathbb R}} \Vert \hat{\ell}-\hat{u}\Vert^2.
\end{equation}

This motivates us to study a projective EDD of the line multiview variety $\mathcal L_{\mathcal C}$ as the number of complex critical points $\ell=(\ell_1,\ldots,\ell_m)$ with $\ell_i\neq 0$ of the Euclidean distance function from $\widehat{\mathcal L}_{\mathcal C}\subset (\mathbb C^{3})^m$ to a general point  $\hat{u}=(\hat{u}_1,\ldots,\hat{u}_m)\in(\mathbb R^{3})^m$. The next lemma shows that we may assume $u_1^Tu_1 = \cdots = u_m^Tu_m$.
\begin{lemma}
Let $\hat{\ell}=(\hat{\ell}_1,\ldots,\hat{\ell}_m)\in\widehat{\mathcal L}_{\mathcal C}$ be a critical point for the Euclidean distance function to the point $\hat{u}=(\hat{u}_1,\ldots,\hat{u}_m)\in(\mathbb R^{3})^m$ and $\lambda_1,\ldots,\lambda_m\neq 0$. Then, $\lambda.\hat{\ell}:=(\lambda_1\hat{\ell}_1,\ldots,\lambda_m\hat{\ell}_m)\in\widehat{\mathcal L}_{\mathcal C}$
is a critical point for the Euclidean distance function to $\lambda.\hat{u}:=(\lambda_1\hat{u}_1,\ldots,\lambda_m\hat{u}_m)\in(\mathbb R^{3})^m$.
\end{lemma}
\begin{proof}
Because $\widehat{\mathcal L}_{\mathcal C}$ is the cone over the multiprojective variety $\mathcal L_{\mathcal C}$, the tangent space is closed under entrywise scalar-multiplication:
$T:=T_{\hat{\ell}}\widehat{\mathcal L}_{\mathcal C} = T_{\lambda.\hat{\ell}}\widehat{\mathcal L}_{\mathcal C}$.
Let $v=(v_1,\ldots,v_m)\in T$ be a tangent vector. We have
$$v^T(\lambda.\hat{\ell} - \lambda.\hat{u}) = \sum_{i=1}^m v_i^T(\lambda_i\hat{\ell}_i-\lambda_i \hat{u}_i) = \sum_{i=1}^m (\lambda_iv_i)^T(\hat{\ell}_i- \hat{u}_i) =0,$$%
because
$(\lambda_1v_1,\ldots,\lambda_mv_m)\in T$. 
\end{proof}
\begin{remark} This discussion applies to any subvariety in a product of projective spaces. We are unaware of any reference that defines the EDD of a multiprojective variety.
\end{remark}
To compute the EDD of the line multiview variety we turn the computation of critical points of the optimization problem (\ref{optimization_problem}) into the problem of solving a system of polynomial equations. For this, we proceed as follows. Recall from (\ref{def_tau}) the map
$\tau:\mathbb C^{2\times 2} \to \mathbb G$, which parameterizes a Zariski open subset of $\mathbb G$.
Take a line $L=\tau([\begin{smallmatrix} v_{11}&v_{12}\\
v_{21}&v_{22}\end{smallmatrix}])$. Then, $\ell=(\ell_1,\ldots,\ell_m)=\Upsilon_{\mathcal C}(L)$ is given by
\begin{equation}\label{equation_EDD1}
\ell_i = t_i\kappa_i(V),\quad \text{where}\quad \kappa_i(V):=\big(C_i \begin{bmatrix}1&0&v_{11}&v_{12}\end{bmatrix}^T\big)\ \times\ \big(C_i \begin{bmatrix}0&1&v_{21}&v_{22}\end{bmatrix}^T\big) \in\mathbb C^3,
\end{equation}
where $t_i$ is an extra variable and $\times$ denotes the cross-product in $\mathbb C^3$ (recall that the cross product $u:=v\times w$ satisfies $u^Tv=u^Tw=0$). The variables $t_1,\ldots,t_m$ model the cone over the product of projective spaces: $\kappa_i$ defines a point in the projective class of $\ell_i$ and varying $t_i$ over $\mathbb C$ gives the line through $\kappa_i$ and the origin. Fix a general $u=(u_1,\ldots,u_m)\in(\mathbb C^{3})^m$, and let us define
$$f_{u, \mathcal C}(t,V) = (u -(t_i\kappa_i)_{i=1}^m)^T(u -(t_i\kappa_i)_{i=1}^m).$$
This is a polynomial in the $4+m$ variables $v_{11} v_{12}, v_{21}, v_{22}, t_1,\ldots,t_m$.
The EDD of the line multiview variety $\mathcal L_{\mathcal C}$ is then the number of complex zeros such that $t_i\neq 0$ of the following system of $m+4$ polynomial equations in $m+4$ variables for general $u\in\mathbb C^{3m}$:
\begin{align}\label{EDD_equations}
\partial f_{u, \mathcal C}/\partial v_{11} =
\partial f_{u, \mathcal C}/\partial v_{12} =
\partial f_{u, \mathcal C}/\partial v_{21} =
\partial f_{u, \mathcal C}/\partial v_{22} &= 0\\
\partial f_{u, \mathcal C}/\partial t_1 = \cdots =\partial f_{u, \mathcal C}/\partial t_m &= 0.\nonumber
\end{align}
This system of polynomials 
is the gradient of $f_{u, \mathcal C}$ with respect to the $m+4$ variables $v_{11} v_{12}, v_{21}, v_{22}$ and $t_1,\ldots,t_m$. It will have solutions with $t_i=0$, but these give singular points on $\widehat{\mathcal L}_{\mathcal C}$ and do not correspond to points in a product of projective spaces. This is why we don't count them. Furthermore, for  $m= 3$ we have to sort out one potential singular point, while for $m\geq 4$ all computed solutions give smooth points on $\mathcal L_{\mathcal C}$ by Corollary~\ref{cor_smoothness}.

\begin{lemma}
The EDD of the line multiview variety $\mathcal L_{\mathcal C}$ is constant on a Zariski open set of $m$-tuples of cameras~$\mathcal C=(C_1,\ldots,C_m)\in(\mathbb C^{3\times 4})^m$.
\end{lemma}
\begin{proof} Let $\varepsilon_{\mathcal C}$ denote the EDD of $\mathcal L_{\mathcal C}$.
Let $\mathcal E_{\mathcal C}\subset \mathbb C^{3m}\times \mathbb C^{3m}$ be the \emph{ED correspondence} of $\widehat{\mathcal L}_{\mathcal C}$, as defined in \cite[Section 4]{draisma2016euclidean}. By \cite[Theorem 4.1]{draisma2016euclidean}, we have a projection $\pi:\mathcal E_{\mathcal C}\to (\mathbb C^{3})^m$,
such that for a general $u\in (\mathbb C^{3})^m$
the fiber $\pi^{-1}(u)$ is finite and consists of $\varepsilon_{\mathcal C}$ points. Consider now
$$\mathcal E:=\bigcup_{\mathcal C \in(\mathbb C^{3\times 4})^m} (\mathcal E_{\mathcal C} \times \{\mathcal C\}).$$
By Theorem \ref{thm: main2} the equations for $\mathcal L_{\mathcal C}$ are polynomial in $\mathcal C$, which implies that $\mathcal E$ is a variety. We define the projection $\Pi:\mathcal E\to (\mathbb C^{3})^m\times (\mathbb C^{3\times 4})^m$, $(q,\mathcal C)\mapsto (\pi(q),\mathcal C)$. By construction, for general $(u,\mathcal C)$, the fiber $\Pi^{-1}(u,\mathcal C)$ has cardinality $\varepsilon_{\mathcal C}$. From Noether's Normalization Lemma \cite[Chapter~10]{Gathmann:CommAlgebra} it follows
that there exists a system of polynomial equations $G_{u,\mathcal C}(\ell)$ in $\ell$ whose coefficients depend polynomially on $u$ and $\mathcal C$, such that
$\Pi^{-1}(u,\mathcal C)=\{(\ell,u,\mathcal C)\in\mathcal E \mid G_{u,\mathcal C}(\ell) = 0\}$; see, e.g., \cite[Remark 4.13]{breiding2021algebraic}. This implies that there exists a proper algebraic subvariety $\Sigma \subset (\mathbb C^3)^m\times (\mathbb C^{3\times 4})^m$ and a number $N$, such that the number of zeros of $G_{u,\mathcal C}$ is~$N$ when $(u,\mathcal C)\not\in\Sigma$; see, e.g., \cite[Theorem 7.1.1]{Sommese:Wampler:2005}. Therefore, $\Pi^{-1}(u,\mathcal C)$ is constant on a Zariski open subset of $(\mathbb C^3)^m\times (\mathbb C^{3\times 4})^m$.
\end{proof}

To get an idea of the EDD of the multiview variety for general cameras, we solve the system of equations 
above using \texttt{HomotopyContinuation.jl} \cite{HC.jl}. We certify the outcome of the computation with the certification method based on interval arithmetic implemented in \texttt{HomotopyContinuation.jl}; see \cite{breiding2021certifying}. The algorithm implemented in \texttt{HomotopyContinuation.jl} provides intervals for the real and imaginary  parts of every variable, such that the true solution provably lies in these intervals. This makes it possible to certify that $t_i\neq 0$ (by checking if zero is contained in these intervals).
As explained in \cite[Section 1.1]{breiding2021certifying} we get provably lower bounds for the EDD. This is summarized in the next theorem.

\begin{theorem}\label{EDD_thm}
Let~$\mathrm{EDdeg}(m)$ denote the EDD of the line multiview variety $\mathcal L_{\mathcal C}$ for a general collection of cameras $\mathcal C=(C_1,\ldots, C_m)$. Then:
\begin{align*}
\mathrm{EDdeg}(3)&\geq 74; \\
\mathrm{EDdeg}(4)&\geq 934; \\
\mathrm{EDdeg}(5)&\geq 3651;\\
\mathrm{EDdeg}(6)&\geq 9887; \\
\mathrm{EDdeg}(7)&\geq 21807; \\
\mathrm{EDdeg}(8)&\geq  42073; \\
\mathrm{EDdeg}(9)&\geq 73883;\\
\mathrm{EDdeg}(10)&\geq 120923; \\
\mathrm{EDdeg}(11)&\geq 187406.
\end{align*}
\end{theorem}
\begin{remark}
For $m\leq 6$ we could use the polyhedral homotopy algorithm by Hubert and Sturmfels \cite{HuSt95} implemented in \texttt{HomotopyContinuation.jl}. For $m\geq 7$ this did not work anymore, and we had to use monodromy \cite{DHJLLS}.
\end{remark}
The numbers of Theorem \ref{EDD_thm} compare to the formula for the point multiview variety (\ref{EDD_points}) as follows. For $m=2,3,4,5$, respectively, we get EDDs $6,47,148,336$, respectively, for the point multiview variety.

\section{Sensitivity}
In the previous sections, we have approached the line multiview variety from the perspective of algebraic geometry, studying its algebraic properties. In this section, we want to consider our setup from the point of view of numerical analysis.

We restrict here to \emph{real} data because data in computer vision is usually given as point points and lines with real coordinates, not complex. If $\mathcal C=(C_1,\ldots, C_m)\in(\mathbb R^{3\times 4})^m$ is a collection of real camera matrices, we have the real version of the camera map from (\ref{def_camera_map}): $\Upsilon_{\mathcal C}: \mathbb G_{\mathbb R}\dashrightarrow (\mathbb P^2)^m,$
which takes a real line $L$ in three-space to an $m$-tuple $\ell$ of real lines in two-space. In applications of computer vision one often wants to go the other way and reconstruct $L\in \mathbb G_{\mathbb R}$ from the tuple $\ell = \Upsilon_{\mathcal C}(L)$. This problem is called a \emph{triangulation} problem, a classic but fundamental problem in computer vision. Obtaining fast and accurate triangulation is at the core of many research efforts. We now start an investigation of the sensitivity of the triangulation problem for lines.

We have shown in Propoposition \ref{properties_of_LMV} that generically $\Upsilon_{\mathcal C}$ is identifiable, meaning that for general $L\in\mathbb G_{\mathbb R}$ we have $\Upsilon_{\mathcal C}^{-1}(\Upsilon_{\mathcal C}(L)) = \{L\}$.
This shows that, in principle, the triangulation problem for lines is theoretically feasible.
But this does not imply that it is numerically feasible -- small errors in the data $\ell$, for instance as a result of noisy measurements during the image formation process, could imply large errors in the solution $L$. To estimate this sensitivity we make the following numerical experiments.

We consider a tuple of real lines $\ell = \Upsilon_{\mathcal C}(L)$, $L\in\mathbb G_{\mathbb R}$. Adding noise to $\ell$ gives $u$ near $\ell$. To reconstruct $L$ from~$u$ we solve the distance minimization problem
$\min_{L\in\mathbb G_{\mathbb R}} \,d(\Upsilon_{\mathcal C}(L),u)^2$. We solve this optimization problem by computing the zeros of the system of equations (\ref{EDD_equations}) using \texttt{HomotopyContinuation.jl}  \cite{HC.jl}. This gives $L_0\in\mathbb G_{\mathbb R}$.
To estimate the sensitivity we then record the number
\begin{equation}\label{def_empirical_error}
e_\mathrm{lines} := \log_{10} \frac{\mathrm{dist}(L,L_0)}{d(\ell,u)}.
\end{equation}
The interpretation of $e_\mathrm{lines}$ is that the error in the data $\ell$ gets amplified by a factor of $10^{e_\mathrm{lines}}$. Notice that we rely on a choice of measuring distances: for distances in $(\PP^2)^m$ we use the distance in (\ref{def_metric_Pn_product}) and for distances in the Grassmannian we use (\ref{def_distance}). These are not canonical choices.

\begin{figure}
\begin{center}

\includegraphics[width = 0.425\textwidth]{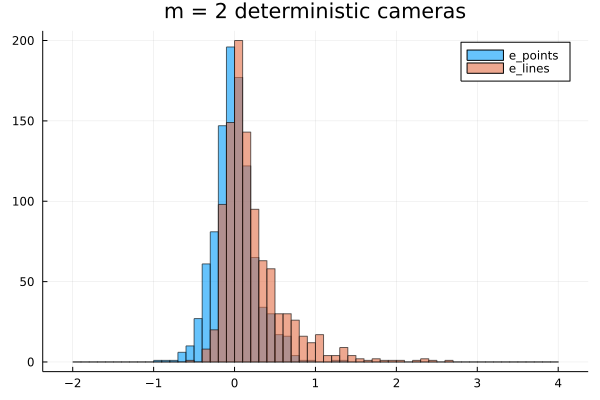}
\hfill
\includegraphics[width = 0.425\textwidth]{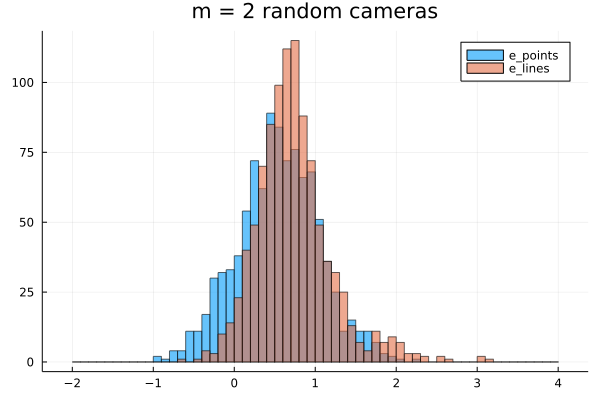}

\includegraphics[width = 0.425\textwidth]{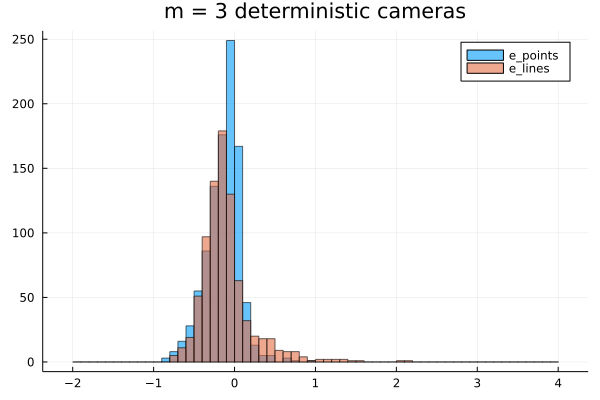}
\hfill
\includegraphics[width = 0.425\textwidth]{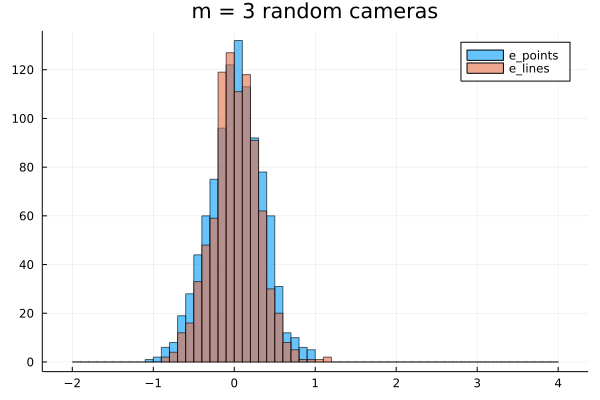}

\caption{\label{fig3} The four histograms show the outcomes of the experiments described in this section. The two pictures on the top show the experiments for $m=2$ cameras. The pictures on the bottom show the experiments for $m=3$ cameras. The left pictures show the empirical distribution of $e_\mathrm{lines}$ (defined in (\ref{def_empirical_error})) and $e_{\mathrm{points}}$ (defined in (\ref{def_empirical_error2})) for the the cameras in (\ref{cameras}) and for 1000 randomly chosen points (blue) and for 1000 randomly chosen lines (orange). The right pictures show the same experiment but for a setup of cameras chosen by sampling independent matrices with i.i.d.\ standard normal entries. The plots we created using \texttt{Plots.jl} \cite{plots}.}
\end{center}
\end{figure}

In our experiment, we take $m=2$ and $m=3$ cameras. In the first experiment, shown in the left pictures in Figure \ref{fig3}, we take the camera matrices
\begin{equation}\label{cameras}
C_1=\begin{bmatrix}0 & 1 & 0 & 0 \\0&0 & 1 & 0 \\0&0 & 0 & 1 \end{bmatrix},\quad
C_2=\begin{bmatrix}0 & 1 & 0 & 0\\0 & 1 & 1 & 0\\0 & 0 & 0 & 1\end{bmatrix},\quad
C_3=\begin{bmatrix}1 & 0 & 0 & 0\\0 & 1 & 1 & 0\\0 & 0 & 0 & 1\end{bmatrix}
\end{equation}
(in the case $m=2$ we take $C_1$ and $C_2$). In the second experiment, which is shown in the right pictures in Figure \ref{fig3}, we take randomly chosen real cameras by sampling independent $3\times 4$ matrices with i.i.d.\ real standard Gaussian entries.
In both settings we sample independently 1000 points~$L\in\mathbb G_{\mathbb R}$ by sampling~4 i.i.d.\ real standard Gaussian random variables $ v_{11},v_{12},
v_{21},v_{22}$ and setting~$L=\tau([\begin{smallmatrix} v_{11}&v_{12}\\
v_{21}&v_{22}\end{smallmatrix}])$. Then, we compute $\ell=(\ell_1,\ldots,\ell_m)= \Upsilon_{\mathcal C}(L)$. Every $\ell_i$ is given to us in terms of three coordinates $\hat \ell_i\in\mathbb R^3$, so that~$\pi(\hat \ell_i)=\ell_i$. We sample $x_1,\ldots,x_m$ independently, where $x_i$ is a point chosen uniformly in the sphere of radius $\epsilon\Vert \hat\ell_i\Vert$, $\epsilon = 10^{-12}$, and then set $u_i=(\hat\ell_i + x_i)/ \Vert\hat\ell_i + x_i\Vert$. We use $u=(u_1,\ldots,u_m)\in(\mathbb R^3)^m$ for setting up the system of polynomial equations (\ref{EDD_equations}).

We do a similar experiment in the case of point reconstruction, where we aim to reconstruct a real point $[1:P_1:P_2:P_3]\in\mathbb P^3$ from $m$ real images $[1:p_1^{(i)}:p_2^{(i)}]\in\mathbb P^2$, $1\leq i\leq m$, where the $i$th image is taken using camera $C_i$. In this setting, there is a natural choice of affine chart setting the first coordinate equal to $1$. Let us write $P:=(P_1,P_2,P_3)\in\mathbb R^3$ and $p^{(i)}:=(p_1^{(i)},p_2^{(i)})\in\mathbb R^2$. For every~$i$ we consider $q^{(i)} := p^{(i)} + x_i$,
where $x_i$ is a point chosen uniformly in the sphere of radius $\epsilon\Vert p^{(i)}\Vert$, $\epsilon = 10^{-12}$, such that $x_1,\ldots,x_m$ are independent. For reconstructing $P$ we can set up a system of polynomial equations to minimize $\sum_{i=1}^m \Vert p^{(i)} - q^{(i)}\Vert^2$
using first order optimality conditions; see \cite{computer-vision2022} for a detailed explanation how to implement this in \texttt{HomotopyContinuation.jl}.
Let $Q\in\mathbb R^3$ be the computed minimizer. We measure the (relative) error by
\begin{equation}\label{def_empirical_error2}
e_\mathrm{points} := \log_{10} 
\Bigg(\frac{\Vert P -Q\Vert}{\sqrt{\sum_{i=1}^m\Vert p^{(i)} -q^{(i)}\Vert^2}} \; 
\frac{\sqrt{\sum_{i=1}^m\Vert p^{(i)}\Vert^2}}{\Vert P\Vert}\Bigg) = 
\log_{10} \frac{\Vert P -Q\Vert}{\epsilon \Vert P\Vert}
\end{equation}
(we measure relative errors, because floating point arithmetic introduces relative perturbations).

The pictures in Figure \ref{def_empirical_error} show the empirical distribution of the empirical errors (\ref{def_empirical_error}) and (\ref{def_empirical_error2}) in histograms. 
The distributions for points and lines seem similar. Of course, other distance measures might imply different distributions, but it is not unreasonable to expect similar sensitivity properties for both points and lines reconstruction problems.
The code for our experiments is attached to the~\texttt{arXiv} version of this article.

\section{Conclusions}
{Given $m$ pinhole cameras we give a set of polynomials cutting out the line multiview variety, that is, we give polynomial constraints satisfied by 2-dimensional line correspondences that can be reconstructed to a~3-dimensional line. Our results extend the description of the line multiview variety for 3 views done by Kileel~\cite{Kileel_Thesis} and also consider the case of cameras not in generic position, that is, when more than 4 cameras are collinear. 
In addition to these polynomial equations, we study some smooth and singular points in the line multiview variety and explore numerically the sensitivity of line reconstruction to noise in the data. 
From this work, there are natural research questions that we would like to pursue in the future. 

We aim to study the ideal of the multiview variety with the goal of finding generators and computing a Gröbner basis for this ideal. Moreover, we aim to explore the ED degree of the line multiview variety in a formal setting, and we aim to study sensitivity systematically by analyzing condition numbers. For the point triangulation problem, this was initiated in \cite[Section~9]{BV2021}. Recently, Fan, Kileel, and Kimia studied the condition number associated to another problem in computer vision called resectioning  \cite{fan2021instability}.

We hope that this algebraic study of line correspondences in $m$ views allows for the creation and implementation of robust reconstruction algorithms, and for the improvement of the noise correction of the currently used algorithms. }


\bibliographystyle{alpha}
\bibliography{VisionBib}
\newpage

\end{document}

%% file: StartingPackage.tex
\usepackage{tabularx}
\usepackage{amsmath}
\usepackage{graphicx}
\usepackage{amsmath}
\usepackage{mathtools}
\usepackage{amsfonts}
\usepackage{amssymb}
\usepackage{amsthm}
\usepackage{cite}
\usepackage[final]{hyperref}
\usepackage{marvosym}
\usepackage{mathrsfs}
\usepackage{enumitem}
\usepackage{dsfont}
\usepackage{tikz-cd}
\tikzset{
    circ/.style={circle, fill=black, inner sep=2pt, node contents={}}
}
\usepackage{comment}
\usepackage[utf8]{inputenc}
\usepackage{changepage}
\usepackage[format=plain, font=footnotesize]{caption}
\hypersetup{
	colorlinks=true,       
	linkcolor=orange,        
	citecolor=purple,        
	filecolor=magenta,     
	urlcolor=blue
}
\usepackage{blindtext}

\newcommand{\PP}{\mathbb{P}}

\newcommand{\CC}{\mathbb{C}}

\theoremstyle{plain}
\newtheorem{define}{Definition}[section]
\theoremstyle{plain}
\newtheorem{theorem}[define]{Theorem}
\theoremstyle{plain}

\theoremstyle{definition}

\theoremstyle{definition}

\theoremstyle{definition}

\theoremstyle{definition}

\theoremstyle{definition}

\theoremstyle{definition}

\theoremstyle{definition}

\theoremstyle{plain}
\newtheorem{proposition}[define]{Proposition}
\theoremstyle{plain}
\newtheorem{lemma}[define]{Lemma}
\theoremstyle{plain}
\newtheorem{corollary}[define]{Corollary}
\theoremstyle{remark}
\newtheorem*{remark}{Remark}
\theoremstyle{plain}